\documentclass[12pt]{article}
\usepackage{graphicx}
\usepackage{amssymb}
\usepackage{latexsym, amssymb, amsfonts, amsmath, amsthm, graphics, graphicx, subfigure, bbm, stmaryrd, xcolor}
\usepackage{epstopdf}

 \usepackage{authblk}

\usepackage[margin=2cm]{geometry}

\usepackage{enumerate}

\allowdisplaybreaks[4]

\newcommand\floor[1]{\lfloor#1\rfloor}

\topmargin by -\headsep \textheight 8.9in \oddsidemargin 0pt
\evensidemargin \oddsidemargin \marginparwidth 0.5in \textwidth
6.5in

\newcommand{\D }{\mathcal{D}}

\newcommand{\CalC}{\mathcal{C}}

\newcommand{\RR}{{\mathbb R}}

\newcommand{\ZZ}{{\mathbb Z}}

\newcommand{\EE }{{\mathbb E}}

\newcommand{\PP}{{\mathbb P}}

\newcommand{\dsum}{\displaystyle\sum}
\newcommand{\dlim}{\displaystyle\lim}
\newcommand{\dliminf}{\displaystyle\liminf}

\newtheorem{theorem}{Theorem}[section]
\newtheorem{proposition}[theorem]{Proposition}

\newtheorem{lemma}{Lemma}[subsection]
\newtheorem{definition}[theorem]{Definition}

\numberwithin{equation}{subsection}

\makeatletter
\renewcommand\section{\@startsection {section}{1}{\z@}%
    {-3.5ex \@plus -1ex \@minus -.2ex}%
    {2.3ex \@plus.2ex}%
    {\normalfont\fontsize{18}{19}\bfseries}}
\makeatother

\makeatletter
\renewcommand\subsection{\@startsection {subsection}{1}{\z@}%
    {-1.5ex \@plus -1ex \@minus -.2ex}%
    {1.3ex \@plus.2ex}%
    {\normalfont\fontsize{13}{14}\bfseries}}
\makeatother

\makeatletter
\newcommand\xleftrightarrow[2][]{%
  \ext@arrow 9999{\longleftrightarrowfill@}{#1}{#2}}
\newcommand\longleftrightarrowfill@{%
  \arrowfill@\leftarrow\relbar\rightarrow}
\makeatother

\title{\bf The acceptance profile of invasion percolation at $p_c$ in two dimensions}

\author{Bounghun Bock$^*$}
\author{Michael Damron\footnote{School of Mathematics, Georgia Institute of Technology, 686 Cherry St., Atlanta, GA 30332 \\ Email: mdamron6@gatech.edu,~bbock3@gatech.edu}}
\affil{Georgia Tech}

\date{\today}

\begin{document}

\baselineskip20pt

\maketitle

\begin{abstract}
Invasion percolation is a stochastic growth model that follows a greedy algorithm. After assigning i.i.d. uniform random variables (weights) to all edges of $\mathbb{Z}^d$, the growth starts at the origin. At each step, we adjoin to the current cluster the edge of minimal weight from its boundary. In '85, Chayes-Chayes-Newman studied the ``acceptance profile'' of the invasion: for a given $p \in [0,1]$, it is the ratio of the expected number of invaded edges until time $n$ with weight in $[p,p+\text{d}p]$ to the expected number of observed edges (those in the cluster or its boundary) with weight in the same interval. They showed that in all dimensions, the acceptance profile $a_n(p)$ converges to one for $p<p_c$ and to zero for $p>p_c$. In this paper, we consider $a_n(p)$ at the critical point $p=p_c$ in two dimensions and show that it is bounded away from zero and one as $n \to \infty$.
\end{abstract}

\section{Introduction}

\subsection{The model} \label{subsection: The model}

We begin with the definition of invasion percolation. It is a stochastic growth model introduced independently by two groups (\cite{CKLW} and \cite{LB}) and is a simple example of self-organized criticality. That is, although the model itself has no parameter, its structure on large scales resembles that of another critical model: critical Bernoulli percolation.

Let $\ZZ^2$ be the two-dimensional square lattice and $\mathcal{E}^2$  be the set of nearest-neighbor edges. For a subgraph $G=(V,E)$ of $(\ZZ^2, \mathcal{E}^2)$, we define the outer (edge) boundary of $G$ as
\begin{equation*} 
\partial G := \{e=\{x,y\} \in \mathcal{E}^d :  e \notin E, \text{but}\; x \in V\; \text{or}\; y \in V\}.
\end{equation*}

Assign i.i.d uniform random $[0,1]$ variables $(\omega(e))$ to all bonds $e\in \mathcal{E}^2$. The \textit{invasion percolation cluster} (IPC) $G$ can be defined as the limit of an increasing sequence of subgraphs $(G_n)$ as follows. The graph $G_0$ has only the origin and no edges. Once $G_i=(V_i, E_i)$ is defined, we select the edge $e_{i+1}$ that minimizes $\omega(e)$ for $e \in \partial G_i$, take $E_{i+1} =E_i \cup \{e_{i+1}\}$ and let $G_{i+1}$ be the graph induced by the edge set $E_{i+1}$. The graph $G_i$ is called the invaded region at time $i$, and the graph $G = \cup_{i=0}^\infty G_i$ is called the \textit{invasion percolation cluster} (IPC). 

The first rigorous study of invasion percolation was done in '85 by Chayes-Chayes-Newman \cite{StochasticGeometry}, who took a dynamical perspective: their questions were related to the evolution of the graph $G_n$ as $n$ increases. In the '90s and '00s, results focused on a more static perspective: properties of the full invaded region. For example, the fractal dimension of $G$ was determined \cite{MR1316507} along with finer properties of $G$ like relations to other critical models \cite{MR1981994}, analysis of the pond and outlet structure \cite{MR2800910, MR2573559}, and scaling limits \cite{garban}.

In this paper, we return to the earlier dynamical perspective and study the ``acceptance profile'' of the invasion, introduced in \cite{MR725616}. Roughly speaking, the acceptance profile $a_n(p)$ at value $p$ and time $n$ is the ratio
\[
a_n(p) = \frac{\text{expected number of bonds invaded with weight in }[p,p+\text{d}p]}{\text{expected number of bonds observed with weight in }[p,p+\text{d}p]},
\]
where both the numerator and denominator are computed until time $n$, and a bond is observed by time $n$ if it is either invaded by time $n$ or is on the boundary of the invasion at time $n$. In \cite[Theorems~4.2, 4.3]{StochasticGeometry}, it is shown that for general dimensions, if $p < \pi_c$ (a certain critical threshold for independent percolation), one has $a_n(p) \to 1$ as $n \to \infty$ and if $p > \bar{p}_c$ (another threshold value with $\bar{p}_c \geq \pi_c$), one has $a_n(p) \to 0$ as $n \to \infty.$ Since publication of that paper, it has been established that $\bar{p}_c = \pi_c = p_c$, where $p_c$ is the standard critical value for independent percolation. Since $p_c=1/2$ in dimension 2, we have
\[
\lim_{n \to \infty} a_n(p) = \begin{cases}
1 & \quad\text{if } p < 1/2 \\
0 & \quad\text{if } p > 1/2.
\end{cases}
\]
This result means that when $p<p_c$, all observed edges with weight near $p$ are invaded relatively quickly, whereas for $p>p_c$, observed edges with weight near $p$ are never invaded (for $n$ large).

The case $p=p_c$ was left open in \cite{StochasticGeometry}, and it is this case we study here. It would be very interesting to establish the existence of $\lim_{n \to \infty} a_n(p_c)$, which by the following main theorem, would be a number in $(0,1)$.
\begin{theorem} \label{thm: main theorem} In two dimensions, where $p_c=1/2$,
\[
 0 < \liminf_{n \rightarrow \infty} a_n(p_c) \leq \limsup_{n \rightarrow \infty} a_{n} (p_c) < 1.
\]
\end{theorem}
This theorem roughly states that when $n$ is large, at least $c\epsilon$ fraction of invaded edges have weight in $(p_c,p_c+\epsilon]$, whereas at least $c\epsilon$ fraction of observed edges with weight in this interval are not yet invaded. To prove this result, we will need to study detailed properties of the invaded region at time $n$, which can be quite different than those of the full invaded region.

In the physics literature, the acceptance profile was considered earlier, in work of Wilkinson-Willemsen \cite{MR725616}. There, it was loosely defined as $a(r)$, the ``number of random numbers in the interval $[r,r+\text{d}r]$
which were accepted into the cluster, expressed as a fraction of the number of random numbers in that range which became available.'' It was noted in that paper that the acceptance profile appears to approach a step function with jump at $p_c$, and that for values of $p$ near $p_c$, ``there is a transition region in which some numbers are accepted and some rejected.'' (See \cite[Fig.~2]{MR725616}.) This observation, although for a different version of the acceptance profile (there is no expected value as in the acceptance profile of Chayes-Chayes-Newman that we work with), is consistent with our main theorem. The step function property of the profile has later been used to estimate numerical values of $p_c$ (see, for example, \cite{WW83b}).

In the next section, we give a rigorous definition of the acceptance profile along with the results of \cite{StochasticGeometry}. To do this, we will also introduce the standard Bernoulli percolation model.


\subsection{Acceptance Profile}

To define the acceptance profile, we use the notations of \cite{StochasticGeometry}. Let $I_n \in \mathcal{E}^2$ be the invaded bond at time $n \geq 1$ and let $x_n$ be the random weight of $I_n$ (the weight $\omega(I_n)$). For any $y \in [0,1]$, define $X_n(y)$ as the indicator that $x_n \leq y$:
 \[X_n(y) = \left\{
  \begin{array}{lr}
    1 & \text{if} \;\; x_n \leq y\\
    0 &  \text{otherwise}.
  \end{array}
\right.
\]

Let $R_n$ be the random number of new bonds which must be checked after the invasion of $I_n$ (that is, $R_0 = 4$, $R_1 = 3$, and $R_n$ is the number of boundary edges of $G_n$ that were not boundary edges of $G_{n-1}$) and define $L_n:=\dsum_{j=0}^n R_j$ to be the total number of checked bonds until the invasion of $I_n$. Clearly, $n \leq L_n \leq 4n$. Denote by $v_n$ the value of the $n^{th}$ checked bond. (Here we can enumerate the checked edges counted in $R_n$ in any deterministic fashion.) Set $V_n(y)$ to be the indicator that $v_n \leq y$:
\begin{equation*}
V_n(y) = \left\{
  \begin{array}{lr}
    1 & \text{if} \;\; v_n  \leq y\\
    0 &  \text{otherwise}.
  \end{array}
\right.
\end{equation*}
Then the acceptance profile at value $x$ by time $n$ is defined as
\begin{equation} \label{definition of a_n(x)}
a_n(x)=\dlim_{\epsilon \downarrow 0} \frac{\EE \left[ \dsum_{j=1}^n \bigg( X_j(x+\epsilon)-X_j(x) \bigg) \right]}{\EE \left[\dsum_{j=1}^{L_n} \bigg( V_j(x+\epsilon)-V_j(x) \bigg) \right] } .
\end{equation}
It is shown in \cite[Proposition~4.1]{StochasticGeometry} that $a_n(x)$ is an analytic function of $x$.

An alternative representation for the acceptance profile will be useful for us. Let $\tilde{Q}_{n}(x) = \dsum_{j=1}^n X_j(x)$ be the number of invaded edges until time $n$ with weight $\leq x$ and $\tilde{P}_{n}(x) = \dsum_{j=1}^{L_n} V_j(x)$ be the number of checked edges until time $n$ with weight $\leq x$. From \cite[Eq.~(4.3)]{StochasticGeometry}, one has 
\[
\EE[\tilde{P}_n(x)] = x \EE[L_n],
\]
and so we can rewrite \eqref{definition of a_n(x)} as
\begin{equation} \label{eq : modified definition of a_n(x)}
a_n(x) = \dlim_{\epsilon \downarrow 0} \frac{\EE [ \tilde{Q}_n(x+\epsilon) - \tilde{Q}_n(x) ] }{\epsilon \EE[L_n]}.
\end{equation}

Analysis of the IPC and the acceptance profile heavily involves tools from Bernoulli percolation, whose definition depends on a parameter $p \in [0,1]$. We will couple the percolation model to the IPC in the following standard way. For every $e \in \mathcal{E}^2$ and any $p \in [0,1]$, we say that $e$ is $p$-open if $\omega(e) \leq p$; otherwise, we say that $e$ is $p$-closed. Note that the variables $(\mathbf{1}_{\{e \text{ is }p\text{-open}\}})_{e \in \mathcal{E}^2}$ are i.i.d. Bernoulli random variables with parameter $p$. The main object of study in percolation is the connectivity properties of the graph whose edges consist of the $p$-open edges. If $p$ is large, we expect this graph to contain very large (even infinite) components and if $p$ is small we expect it to contain only small components. To formulate these ideas precisely, we say that a path (a finite or infinite sequence of edges $e_1, e_2, \dots$ such that $e_i$ and $e_{i+1}$ share at least one endpoint) is $p$-open if all its edges are $p$-open, and we write $A \xleftrightarrow{\;\;p \;\;} B$ for two sets of vertices $A$ and $B$ if there is a $p$-open path starting at a vertex in $A$ and ending at a vertex in $B$. We also write $u \xleftrightarrow{\;\;p \;\;} v$ for vertices $u,v$ when $A = \{u\}$ and $V = \{v\}$, and we use the term ``$p$-open cluster of $u$'' to refer to the set of vertices $v$ such that $u \xleftrightarrow{\;\;p \;\;} v$. Last, we write $u \xleftrightarrow{\;\;p \;\;}\infty$ to mean that the $p$-open cluster of $u$ is infinite. Given this setup, we define the critical threshold for percolation as
\[
p_c = \sup\{p \in [0,1] : \theta(p) = 0\},
\]
where
\[
\theta(p) = \mathbb{P}(0 \xleftrightarrow{\;\;p \;\;} \infty).
\]
It is known that for all dimensions $d \geq 2$, one has $p_c \in (0,1)$, and for $d=2$, $p_c=1/2$. These facts and more can be seen in the standard reference \cite{grimmett_percolation}.

In addition to $p_c$, there are other critical values that have been used in the past, and these have mostly been shown to be equal to $p_c$. The two that were used in \cite{StochasticGeometry} are
\begin{equation*}
\begin{split}
& \pi_c  = \sup \{p \in [0,1] :  \EE \#\{v : \text{$v$ is in the $p$-open cluster of 0}\} < \infty )\}, \text{ and } \\
& \bar{p}_{c} = \sup \{p \in [0,1] :  \PP (\text{$\exists$ infinite $p$-open path in a half-space)} =0 \}.
\end{split}
\end{equation*}
In this language, and for general dimensions, the theorems of Chayes-Chayes-Newman state that
\begin{equation*} \label{a_n using halfspace}
\dlim_{n \rightarrow \infty} a_n(p) = 
  \begin{cases}
    1 & \text{if} \;\; p < \pi_c\\
    0 &  \text{if} \;\; p > \bar{p_c}.
  \end{cases}
\end{equation*}
Because $\pi_c$ and $\bar{p_c}$ are both known to be equal to $p_c$ (see \cite{MR1068308,MR0115221,MR914958}), this result specifies the limiting behavior of the acceptance profile at all values of $p \neq p_c$. Our main result, Theorem~\ref{thm: main theorem}, shows that in two dimensions, the limiting behavior of $a_n(p_c)$ is different than that of $a_n(p)$ for any other value of $p$: it remains bounded away from zero and one.

\subsection{Notation and outline of the paper}\label{sec: notations}
First we gather some notation used in the paper. For $n \geq 1$ let $B(n) = [-n,n]^2$ be the box of sidelength $2n$, and for $m<n$, let $\text{Ann}(m,n)$ be the annulus $B(n) \setminus B(m)$. We will be interested in connection probabilities from points to boundaries of boxes, so we set
\[
\pi(p,n) = \mathbb{P}(0 \xleftrightarrow{\;\;p \;\;} \partial B(n))\text{ and } \pi(n) = \pi(p_c,n).
\]
Many connection probabilities (or their complements) can be expressed in terms of connections on the dual graph $(\mathbb{Z}^2)^*$. To define it, let $(\mathbb{Z}^2)^* = \left( \frac{1}{2}, \frac{1}{2} \right) + \mathbb{Z}^2$ be the set of dual vertices and let $(\mathcal{E}^2)^*$ be the edges between nearest-neighbor dual vertices. For $x \in \mathbb{Z}^2$ we write $x^* = x + \left( \frac{1}{2}, \frac{1}{2}\right)$ for its dual vertex. For an edge $e \in \mathcal{E}^2$, we denote its endpoints (left, respectively right or bottom, respectively top) by $e_x,e_y \in \mathbb{Z}^2$. The edge $e^* = \left\{ e_x + \left( \frac{1}{2}, \frac{1}{2} \right), e_y - \left( \frac{1}{2}, \frac{1}{2} \right)\right\}$ is called the edge dual to $e$. (It is the unique dual edge that bisects $e$.) A dual edge $e^*$ is called $p$-open if $e$ is $p$-open, and is $p$-closed otherwise. A dual path is a finite or infinite sequence of dual edges such that consecutive edges share at least one endpoint. A circuit (or dual circuit) is a finite path (or dual path) which has the same initial and final vertices. 

For two functions $f(x)$ and $g(x)$ from a set $\mathcal{X}$ to $\RR$, the notation $f(x) \asymp g(x) $ means $\frac{f(x)}{g(x)}$ is bounded away from $0$ and $\infty$, uniformly in $x \in \mathcal{X}$.

In the next section, we give the proof of Theorem~\ref{thm: main theorem}. It is split into three subsections. In Section~\ref{correlation length}, we introduce correlation length and results which are frequently used in two-dimensional percolation. In Section~\ref{sec: lower bound}, we prove the lower bound of Theorem~\ref{thm: main theorem} and in Section~\ref{sec: upper bound}, we prove the upper bound of Theorem~\ref{thm: main theorem}.

\section{Proof of Theorem~\ref{thm: main theorem}}\label{sec: main theorem}


\subsection{Preliminaries} \label{correlation length}

We first introduce the finite-size scaling correlation length (see a more detailed survey in \cite{MR2438816}).  Let
\begin{equation*}
\sigma(n,m,p) = \PP(\text{$\exists$ a $p$-open horizontal crossing of $[0,n] \times [0,m]$}).
\end{equation*}
Here, a horizontal crossing is a path which remains in $[0,n]\times [0,m]$, with initial vertex in $\{0\} \times [0,m]$ and final vertex in $\{n\}\times \{0,m\}$. For any $\epsilon >0$, we set
\begin{equation*} 
L(p, \epsilon) :=  \left\{ 
  \begin{array}{lr}
   \min \{n : \sigma(n, n,p) \leq \epsilon \}  & \text{if} \;\; p < p_c \\
   \min \{n : \sigma(n, n,p)   \geq 1-\epsilon \}  &  \text{if} \;\; p > p_c \\
  \end{array}
\right.
\end{equation*}

$L(p, \epsilon)$ is called the finite-size scaling correlation length and its scaling as $p \to p_c$ does not depend on $\epsilon$, so long as $\epsilon$ is small enough. That is, there exists an $\epsilon_0 > 0$ such that for $\epsilon_1, \epsilon_2 \in (0,\epsilon_0]$, $L(p, \epsilon_1) \asymp L(p, \epsilon_2)$ as $p \to p_c$ \cite[Eq.~(1.24)]{MR879034}. For this reason, we set
\[
L(p) = L(p,\epsilon_0).
\]
Because $L(p) \to \infty$ as $p \to p_c$ \cite[Prop.~4]{MR2438816} and $L(p) \to 0$ as $p \to 0$ or $p\to 1$, the approximate inverses
\begin{equation*}
\begin{split}
& p_n = \sup \{p  > p_c : L(p) > n\}\\
& q_n = \inf   \{q  < p_c : L(q) > n\}\\
\end{split}
\end{equation*}
are well-defined.

Next we list relevant and now standard properties of the correlation length with references to their proofs.

\begin{enumerate}[1.]
\item {} \cite[Thm.~1]{MR879034} For $n \leq L(p) $ and $p \neq p_c$,  \\
\begin{equation}
 \pi(p,n) \asymp \pi(n).    \label{compare for different p}
\end{equation}

 \item {} \cite[Thm.~2]{MR879034} There are positive constants $C_1$ and $C_2$ such that for all $p >p_c$
\begin{equation}
\pi(L(p)) \leq \pi(p, L(p)) \leq C_1 \theta(p) \leq C_1 \pi(p,L(p)) \leq C_2 \pi(L(p)). \label{bound pi for bigger than critical}
\end{equation}


\item{} \cite[Eq.~(2.8)]{MR1981994}   There are positive constants $C_3, C_4$ such that
\begin{equation}
\sigma(2mL(p), mL(p),p)  \geq 1- C_3 \exp(-C_4 m), \;\; \text{for} \;\;m > 1.  \label{2.8 of Jarai}
\end{equation}

\item {} \cite[Eq.~(2.10)]{MR1981994}     There is a constant $D$ such that
\begin{equation}
\lim_{\delta \downarrow 0} \frac{L(p-\delta)}{L(p)} \leq D \;\; \text{for} \;\; p > p_c.  \label{2.10 of Jarai} 
\end{equation}

\item{} \cite[Cor.~3.15]{MR799280} There exists a constant $D_1>0$ such that
\begin{equation}
\frac{\pi(m)}{\pi(n)} \geq D_1 \sqrt{\frac{n}{m}} \;\; \text{for }  m \geq n \geq 1.  \label{bounded D_1}
\end{equation}

\item{} \cite[Prop.~34]{MR2438816} (Arm events). Fix $e = \{ e_x, e_y\}$ and let $A_n^{2,2}$ be the event that $e_x$ and $e_y$ are connected to $\partial B(n)$ by $p_c$-open paths not containing $e$, and ${e_x}^*$ and ${e_y}^*$ are connected to $\partial B(n)^*$ by $p_c$-closed dual paths not containing $e^*$. Note that these four paths are disjoint and alternate. For $n \geq 1$,
\begin{equation}\label{definition of $A_n^{2,2}$}
\begin{split}
& (p_n - p_c) n^2 \PP(A_n^{2,2})  \asymp 1\\
& (p_c-q_n ) n^2 \PP(A_n^{2,2}) \asymp 1.\\
\end{split} 
\end{equation}

\item{} \cite[Sec.~3.2]{MR2438816} (Russo-Seymour-Welsh: RSW)\label{RSW} For every $k, l\geq 1$, there exists $\delta_{k,l} >0$ such that for all $p \in [p_c,p_n]$ \text{(respectively $q \in [q_n, p_c]$)},
\begin{equation*}
\begin{split}
& \PP(\text{$\exists$ a $p$-open (respectively $q$-open) horizontal crossing of $[0,kn] \times [0,ln]) > \delta_{k,l}$}\\
& \PP(\text{$\exists$ a $p$-closed (respectively $q$-closed) horizontal dual crossing of $([0,kn] \times [0,ln])^* > \delta_{k,l}$}.\\
\end{split}
\end{equation*}
In addition, applying the FKG inequality \cite[Ch.~2]{grimmett_percolation}, for all $p \in [p_c,p_n]$ \text{(resp. $q \in [q_n, p_c]$)},
\begin{equation*}
\begin{split}
& \PP \big(\text{Ann($n, kn$) contains a $p$-open (resp. $q$-open) circuit around the origin} \big) > (\delta_{k,k-2})^4\\
& \PP \big(\text{Ann($n, kn)^*$ contains a $p$-closed (resp. $q$-closed) dual circuit around the origin} \big)> (\delta_{k,k-2})^4.\\
\end{split}
\end{equation*}

\item {} \cite{MR1981994,MR1316507}   Let $|\mathcal{S}_n|$ be the number of invaded edges (edges in $G$) inside $B(n)$. Then,
\begin{equation}
\EE|\mathcal{S}_n| \asymp n^2 \pi(n).  \label{Jarai and Zhaing expectation}
\end{equation}
\end{enumerate}

Last, we prove some lemmas that will be helpful in the proof of the main theorem. These lemmas will bound the random variables
\begin{equation*}
\begin{split}
R_n & :=\min\{ k : I_i \subset B(k)\; \text{for}\;  i=1,2,\cdots, n\}  \\
r_n & :=\max \{k : I_i \subset  B(k)^c \;\; \text{for all} \;\; i > n\}.
\end{split}
\end{equation*}
$R_n$ is a radius of the invaded region at time $n$, and $r_n$ is the largest size of box such that the invasion does not change in this box after time $n$.
%

\begin{lemma} \label{lemma: upperbound of R}  There exists a constant $\mathcal{C}_1>0$ such that for all $n \geq 1$ and $C>0$,
\begin{equation*}
\PP(R_{\lfloor C n^2 \pi(n)\rfloor } <n ) \leq \frac{\mathcal{C}_1}{C}. \label{eq: upperbound of R} 
\end{equation*}
\end{lemma}
\begin{proof} The event $\{R_{\lfloor C n^2 \pi(n)\rfloor } < n\}$ implies that $|\mathcal{S}_n| \geq \lfloor C n^2\pi(n)\rfloor$. By Markov's inequality and \eqref{Jarai and Zhaing expectation},
\[
\mathbb{P}(R_{\lfloor Cn^2\pi(n)\rfloor } < n) \leq \mathbb{P}(|\mathcal{S}_n| \geq \lfloor C n^2\pi(n)\rfloor ) \leq \frac{\mathbb{E}|\mathcal{S}_n|}{\lfloor Cn^2\pi(n)\rfloor} \leq \frac{\mathcal{C}_1}{C}.
\]
\end{proof}

\begin{lemma} \label{lemma: upperbound of small r}  For any $\eta_0 >0$, there exists $\mathcal{C}_{2}>0$ such that for any $C \geq \mathcal{C}_{2}$ and $n \geq 1$,
\begin{equation*}
\PP(r_{\lfloor C  n^2 \pi(n)\rfloor} <  2n )  \leq \eta_0 \label{eq: upper bound of small r}
\end{equation*}
\end{lemma}

\begin{proof} For $k,m \geq 1$, we consider the event $D_{k,m}$ defined by the following conditions:
\begin{enumerate}[(i)]
\item There is a $p_c$-open circuit around the origin in Ann$(2^{k+1}, 2^{k+1+\frac{m}{8}})$.
\item There is a $p_{2^{k+1+\frac{m}{4}}}$-closed dual circuit around the origin in Ann$(2^{k+1+\frac{m}{8}}, 2^{k+1+\frac{m}{4}})^*$.
\item There is a $p_c$-open circuit around the origin in Ann$(2^{k+1+\frac{m}{2}}, 2^{k+1+m})$.
\item The circuit from (iii) is connected to infinity by a $p_{2^{k+1+\frac{m}{4}}}$-open path.
\end{enumerate}
(See Figure~\ref{fig: D_{k,m}} for an illustration of $D_{k,m}$.)

\begin{figure} [h] 
\centering
\includegraphics[width=45mm]{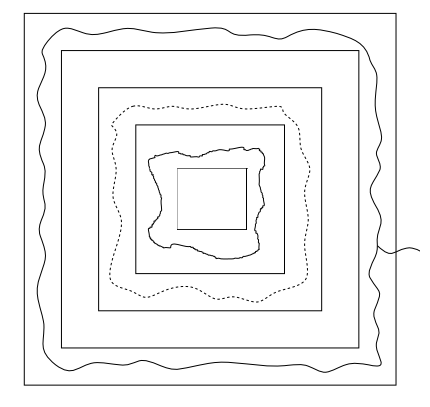}
\caption{ Illustration of the event $D_{k,m}$. The boxes, in order from smallest to largest, are $B(2^{k+1})$, $B(2^{k+1+\frac{m}{8}})$, $B(2^{k+1+\frac{m}{4}})$, $B(2^{k+1+\frac{m}{2}})$ and $B(2^{k+1+m})$. The solid circuit is $p_c$-open, the path to infinity is $p_{2^{k+1+\frac{m}{4}}}$-open, and the dotted path is $p_{2^{k+1+\frac{m}{4}}}$-closed.}
\label{fig: D_{k,m}}
\end{figure}

For $j,k,m \geq 1$, we claim that
\begin{equation}\label{eq: static_implication}
\left( \{R_j \geq 2^{k+1+m} \} \cap D_{k,m}\right) \subset \{r_j \geq 2^{k+1}\}.
\end{equation}
To see why, suppose the left side occurs, and choose $\mathfrak{C}_1$ as a circuit from (i) in the definition of $D_{k,m}$, $\mathfrak{C}_2$ as a circuit from (ii), and $\mathfrak{C}_3$ as a circuit from (iii). Let $n_1$ be the time at which the invasion invades all of $\mathfrak{C}_1$ and for $i=2,3$, let $n_i$ be the first time that the invasion invades an edge from $\mathfrak{C}_i$. Note that $n_1 \leq n_2 \leq n_3 \leq j$. (The last inequality holds because $R_j \geq 2^{k+1+m}$.)

After time $n_3$, the invasion has an unending supply of edges with weight $<p_{2^{k+1+\frac{m}{4}}}$ to invade, so it will never again take an edge with weight larger than that. Furthermore, at time $n_2$, the invasion must take an edge with weight larger than $p_{2^{k+1+\frac{m}{4}}}$. This implies that at some time $n_4 \in [n_2,n_3)$, the invasion invades an outlet: an edge $\hat e$ such that all edges invaded after time $n_4$ have weight $<\omega(\hat e)$. Furthermore, this outlet can be chosen to have weight $\omega(\hat e)>p_{2^{k+1+\frac{m}{4}}} > p_c$.

Directly before time $n_4$, the entire boundary of the invasion (excluding $\hat e$ itself) consists of edges with weight $>\omega(\hat e)$. Since invaded weights beyond time $n_4$ are $<\omega(\hat e)$, none of these boundary edges will ever be invaded. Therefore all invaded edges after time $n_4$ are invaded through $\hat e$. In other words, if $e$ is any edge invaded after time $n_4$, there is a path $P(e)$ connecting $\hat e$ to $e$ consisting of edges with weight $<\omega(\hat e)$ and which are invaded after time $n_4$. It is important to note that $P(e)$ cannot touch $\mathfrak{C}_1$. Indeed, if were to contain an edge $f$ which shared an endpoint with an edge on $\mathfrak{C}_1$ (including the possibility that $f \in \mathfrak{C}_1$), then $f$ would be accessible to the invasion at time $n_1$, and so $f$ would be invaded before time $n_4$, a contradiction.

Finally, to prove \eqref{eq: static_implication}, assume that $r_j < 2^{k+1}$. Then there is some time $j' > j$ at which the invasion invades an edge $e$ in $B(2^{k+1})$. Since $j'>n_4$, there is a path $P(e)$ from $\hat e$ to $e$ as in the preceding paragraph which cannot touch $\mathfrak{C}_1$. This means $\hat e$ is in the interior of $\mathfrak{C}_1$. On the other hand, if $f$ is any edge of $\mathfrak{C}_3$ (necessarily invaded after time $n_4$), the path $P(f)$ connecting $\hat e$ to $f$ would then toucn $\mathfrak{C}_1$, a contradiction. This shows \eqref{eq: static_implication}.


Applying \eqref{eq: static_implication} for $C>0$ and $k,m \geq 1$, we obtain
\begin{equation} 
\PP(r_{\lfloor C  2^{2k} \pi(2^k)\rfloor } <  2^{k+1} ) \leq 2 \max\{\PP(R_{\lfloor C 2^{2k} \pi(2^k)\rfloor} < 2^{k+1+m}), \PP(D_{k,m}^c).  \}\label{eq: upper bound r in lemma}
\end{equation}
As in \cite[proof of Thm.~5]{MR2962082}, the RSW theorem implies that $\mathbb{P}(D_{k,m}^c) \leq e^{-\delta m}$ for some $\delta>0$ uniformly in $k$, so we can fix $m$ so that
\begin{equation} \label{eq: D_m in lemma}
 \PP(D_{k,m}^c)   \leq \frac{\eta_0}{2}  \text{ for all } k \geq 1.
\end{equation}
From Lemma~\ref{lemma: upperbound of R} and the fact that $\pi(n)$ is decreasing in $n$, for any $C \geq (2 \mathcal{C}_1 2^{2+2m})/\eta_0 =: \mathcal{C}_2$, we get
\begin{align*}
\PP \big( R_{\lfloor C 2^{2k} \pi(2^k)\rfloor} < 2^{k+1+m} \big) &\leq \PP \big(R_{ \lfloor (2\mathcal{C}_1/ \eta_0)   2^{2(k+1+m)} \pi(2^{k+1+m})\rfloor} < 2^{k+1+m} ) \leq \frac{\eta_0}{2}.
\end{align*} 
Combining this with \eqref{eq: upper bound r in lemma} and \eqref{eq: D_m in lemma}, we find that for $C \geq \mathcal{C}_2 $,
 \begin{align}
\PP \big( r_{\lfloor C 2^{2k} \pi(2^k)\rfloor} < 2^{k+1} \big)  \leq \eta_0,  \nonumber
\end{align}
and this completes the proof for $n$ of the form $2^k$.

For general $n$, we let $k=k(n) :=\floor{\log_2 n}$, so that
for any $C \geq 4\mathcal{C}_2$,
 \begin{align}
\PP(r_{\lfloor C n^2 \pi(n)\rfloor} <  2n ) 
\leq \PP(r_{ \lfloor \mathcal{C}_2 2^{2(k+1)} \pi(2^{k+1})\rfloor} <  2^{k+2}) \leq \eta_0. \nonumber
\end{align}
\end{proof}

\subsection{Lower bound} \label{sec: lower bound}
In this section, we show that
\begin{equation}\label{eq: to_show_lower_bound}
\liminf_{n \to \infty} a_n(p_c) > 0.
\end{equation}
The first step is to show that it suffices to prove this result for only a certain subsequence of values of $n$. Namely, we first prove that if there exists $\mathcal{C}_3>0$ such that 
\begin{equation}\label{eq: lower_bound_subsequence_reduction}
\dliminf_{n \rightarrow \infty} a_{ \floor{ \mathcal{C}_3 n^2 \pi(n) } } (p_c) >0,
\end{equation}
then \eqref{eq: to_show_lower_bound} follows. 

So assume that \eqref{eq: lower_bound_subsequence_reduction} holds, and let 
\[
k=k(n) := \max \{\ell : \mathcal{C}_3 \ell^2\pi(\ell) \leq n \}.
\] 
(Note that this $k$ actually exists for large $n$ since $\pi(\ell) \geq D_1/\sqrt{\ell}$ by \eqref{bounded D_1}.) Since $\tilde{Q}_n(p_c+\epsilon) - \tilde{Q}_n(p_c)$ is increasing in $n$, 
\[
\tilde{Q}_n(p_c+\epsilon) - \tilde{Q}_n(p_c) \geq \tilde{Q}_{ \floor{ \mathcal{C}_3 k^2 \pi(k) } } (p_c+\epsilon) - \tilde{Q}_{ \floor{ \mathcal{C}_3 k^2 \pi(k) } }(p_c).
\]
So using $n \leq L_n \leq  4n$, we obtain 
\begin{equation*} \label{eq:  first eq in proof corollary}
\begin{split}
a_{n}(p_c) = \lim_{\epsilon \downarrow 0} \frac{\EE[\tilde{Q}_n(p_c+\epsilon) - \tilde{Q}_n(p_c)]}{ \epsilon \EE[L_n]}  &\geq \dlim_{\epsilon \downarrow 0} \frac{\EE[\tilde{Q}_{ \floor{ \mathcal{C}_3 k^2 \pi(k) } } (p_c+\epsilon) - \tilde{Q}_{ \floor{ \mathcal{C}_3 k^2 \pi(k) } }(p_c)]}{\epsilon \EE[L_{\floor{ \mathcal{C}_3  k^2 \pi(k)} } ]} \frac{\EE[L_{\floor{ \mathcal{C}_3  k^2 \pi(k)} } ]}{\EE[L_n]}\\
& \geq a_{\floor{ \mathcal{C}_3 k^2 \pi(k) } }(p_c)   \frac{\mathcal{C}_3  k^2 \pi(k)}{8n}\\
\end{split}
\end{equation*}
Thus to conclude \eqref{eq: to_show_lower_bound} from \eqref{eq: lower_bound_subsequence_reduction}, it suffices to show that $\dliminf_{n \rightarrow \infty} k^2 \pi(k)/n$ is positive. For large $n$, $k(n)$ is greater than 1; therefore,
\[ 
 \frac{k^2 \pi(k)}{n} \geq\frac{k^2 \pi(k)}{\mathcal{C}_3 (k+1)^2 \pi(k+1)} \geq \mathcal{C}_3^{-1}\bigg( \frac{ k }{k+1} \bigg)^2 \geq \frac{1}{4\mathcal{C}_3} > 0.
\] 

To prove \eqref{eq: lower_bound_subsequence_reduction}, we use the following lemma, which bounds the $k^{th}$ moment of the number of edges of the IPC with $(p_c, p_c+\epsilon]$ in $B(n)$.


\begin{lemma} \label{E(Y(B_n))} Let $\mathcal{Y}_{n}(\epsilon)$ be the number of invaded edges in $B(n)$ with $(p_c, p_c+\epsilon]$ for $\epsilon >0$. There exist positive constants $\mathcal{C}_4$ and $\mathcal{C}_5= \mathcal{C}_5(t)$ such that for all $n\geq 1$, 
\[
\liminf_{\epsilon \downarrow 0} \frac{\EE |\mathcal{Y}_n(\epsilon)|}{\epsilon} \geq \mathcal{C}_4 n^2 \pi(n)
\]
and
\[
\EE|\mathcal{Y}_n(\epsilon)|^t \leq \mathcal{C}_5 \left( \epsilon n^2 \pi(n) \right)^t \text{ for all } t \geq 1\text{ and } \epsilon>0.
\]
\end{lemma}

Assuming this lemma for the moment, we can derive \eqref{eq: lower_bound_subsequence_reduction}. From Lemma~\ref{lemma: upperbound of small r}, we can choose $\mathcal{C}_{3}$ so that 
\begin{equation*}
\PP(r_{ \floor{ \mathcal{C}_{3} n^2 \pi(n)} } < 2n) \leq   \frac{\mathcal{C}_4^2}{16\mathcal{C}_5(2)} \text{ for all } n \geq 1.
\end{equation*}
On the event $\{ r_{\floor{ \mathcal{C}_{3} n^2 \pi(n)} } \geq 2n\} $, the IPC in $B(2n)$ does not change after time $\floor{ \mathcal{C}_3 n^2 \pi(n)}$. It follows that the number of invaded edges with $(p_c, p_c+\epsilon]$ until time $\floor{ \mathcal{C}_3 n^2 \pi(n) }$ is at least $\mathcal{Y}_{2n}(\epsilon)$, which is the number of invaded edges with $(p_c, p_c+\epsilon]$ in $B(2n)$. By Lemma~\ref{lemma: upperbound of small r}, Lemma~\ref{E(Y(B_n))} and the Cauchy-Schwarz inequality, if $\epsilon$ is sufficiently small,
\begin{equation*}
\begin{split}
\EE  \left[ \dsum_{j=1}^{ \floor{ \mathcal{C}_3 n^2 \pi(n)} } \bigg( X_j(p_c+\epsilon)-X_j(p_c) \bigg) \right] &\geq  \EE \left[  \dsum_{j=1}^{\floor{\mathcal{C}_3 n^2 \pi(n)} } \bigg( X_j(p_c+\epsilon)-X_j(p_c)  \bigg) \cdot \mathbf{1}_{\{ r_{ \floor{ \mathcal{C}_{3} n^2 \pi(n)} } \geq 2n \} }   \right]  \\
& \geq \EE\left[ \mathcal{Y}_{2n}(\epsilon)  \cdot \mathbf{1}_{\{r_{\floor{ \mathcal{C}_{3} n^2 \pi(n)} } \geq 2n\}}  \right]  \\
&\geq \frac{\mathcal{C}_4}{2} \epsilon (2n)^2 \pi(2n) - \EE\left[ \mathcal{Y}_{2n}(\epsilon) \cdot \mathbf{1}_{\{r_{\lfloor \mathcal{C}_3 n^2 \pi(n) \rfloor} < 2n\}} \right] \\
&\geq \frac{\mathcal{C}_4}{2} \epsilon (2n)^2 \pi(2n) - \sqrt{ \mathcal{C}_5(2) \left( \epsilon (2n)^2 \pi(2n)\right)^2 \frac{\mathcal{C}_4^2}{16\mathcal{C}_5(2)}} \\
&\geq \frac{\mathcal{C}_4 \epsilon (2n)^2 \pi(2n)}{4}.
\end{split}
\end{equation*}
 Combining this with \eqref{eq : modified definition of a_n(x)}, \eqref{bounded D_1}, and the fact that $ n \leq \EE[L_n] \leq 4n$, we obtain
\begin{equation*}
\begin{split}
a_{ \floor{ \mathcal{C}_3 n^2\pi(n) } }(p_c) =\dlim_{\epsilon \downarrow 0} \frac{\EE [ \tilde{Q}_{ \floor{ \mathcal{C}_3 n^2 \pi(n) } } (p_c+\epsilon)-\tilde{Q}_{ \floor{ \mathcal{C}_3 n^2 \pi(n) } } (p_c) ] }{\epsilon \EE[L_{ \floor{ \mathcal{C}_3 n^2 \pi(n) } } ]} 
&\geq \dlim_{\epsilon \downarrow 0} \frac{\mathcal{C}_4 \epsilon (2n)^2 \pi(2n) / 4}{4\epsilon \mathcal{C}_3 n^2 \pi(n)}\\
& = \frac{\CalC_4 D_1}{4 \CalC_3 \sqrt{2}},
\end{split}
\end{equation*}
which is positive uniformly in $n$. This shows \eqref{eq: lower_bound_subsequence_reduction}.

The last step is to prove Lemma~\ref{E(Y(B_n))}.
\begin{proof} [Proof of Lemma~\ref{E(Y(B_n))}] The proof of the upper bound is similar to that of J\'arai \cite[Theorem~1]{MR1981994}, which shows an upper bound for $|\mathcal{S}_n|$ (that result does not involve a condition on the weight $\omega(e)$) so we will omit some details. We will follow that proof, but make the events independent of $\omega(e)$ so that we can insert the condition $\omega(e) \in (p_c,p_c+\epsilon]$.

We will restrict to $n$ of the form $2^K$, as the general result follows from this and monotonicity of $\pi_n$. Let $A_{k}$ be Ann($2^{k}, 2^{k+1})$, and  $\mathcal{Y}_{A_k}$ be the number of IPC edges in Ann($2^k,2^{k+1})$ with the weight in $(p_c, p_c+\epsilon]$. Then, $B(n)=\cup_{k=1}^K A_k$ and $\mathcal{Y}_{n}(\epsilon)=\sum_{k=1}^K \mathcal{Y}_{A_k}$. Define a sequence $p_k(0) > p_k(1)> \cdots > p_c$ as follows. Let $\log^{(0)} k = k$, and let $\log^{(j)} k = \log(\log^{(j-1)} k)$ for $j \geq 1$ if the right-hand side is defined.
For $k > 10$, we define
\begin{equation*}
 \log^* k = \min \{j> 0 : \log^{(j)} k\;\text{ is defined and } \log^{(j)} k \leq 10\}.
\end{equation*}
Then $\log^{(j)}k > 2$, for $j = 0, 1, \cdots, \log^* k$ and $k> 10$. Let
\begin{equation*}
p_k(j) = \inf \left\{ p> p_c : L(p) \leq \frac{2^k}{C_{5} \log^{(j)} k} \right\},\;\; j = 0, 1, \cdots, \log^* k,
\end{equation*} 
where the constant $C_5$ will be chosen later. With \eqref{2.10 of Jarai} and \cite[Eq.~(2.15)]{MR1981994}, we get
\begin{equation} 
C_5 \log^{(j)} k \leq \frac{2^k}{L(p_k(j))} \leq D C_5 \log^{(j)} k \label{D}
\end{equation}
For any fixed $e \subset A_{k}$ we define
\begin{align}
& H_k(j)=\{\text{$\exists$ a $p_k(j)$-open circuit $\D$ around the origin in $A_{k-1}$ and $\D \xleftrightarrow{p_k(j)} \infty$}\} \nonumber\\
& H^e_k(j) = \{ \text{$H_k(j)$ occurs and $\D \xleftrightarrow{\;p_k(j)\;} \infty $ without using the edge $e$}\}. \label{Def of H^e_k(j)}
\end{align}
To give a lower bound for the probability of $H_k(j)$, J\'arai constructed an infinite $p_k(j)$-open path starting from $\partial B(2^k)$ using standard 2D constructions only to the right of $B(2^k)$. (See \cite[Fig 1]{MR1981994}). Similarly, to lower bound the probability of $H_k(j)^e$, we build, in addition to J\'arai's path, an infinite $p_k(j)$-open path starting from $\partial B(2^k)$ in the left of $B(2^k)$. The existence of such disjoint two infinite $p_k(j)$-open paths imply the event $\{\D \xleftrightarrow{\;p_k(j)\;} \infty \text{ without using $e$} \}$ for any fixed edge $e \in A_k$. As in \cite[Eq.~(2.17)]{MR1981994}, we obtain
\begin{equation} \label{lowerbounded of $H_k^e(j)$}
J_k(j) \cap \bigg( \bigcap_{m=0}^\infty J_{k,L}^m(j) \bigg) \cap \bigg( \bigcap_{m=0}^\infty J_{k,R}^m(j) \bigg) \subseteq H^e_k(j)
\end{equation}
where for $m \geq 0$,
\begin{equation*}
\begin{split}
&J_k=\{\text{$\exists$ a $p_k(j)$-open circuit in $A_{k-1}$}\} \\
&J_{k,R}^m=J_{k,R}^{m,h} \cap J_{k,R}^{m,v}, \; \text{and} \; J_{k,L}^m=J_{k,L}^{m,h} \cap J_{k,L}^{m,v}\\
&J_{k,R}^{m,h}=\{\text{$\exists$ a $p_k(j)$-open horizontal crossing of  $[2^{k-2+m},2^{k+m}] \times [-2^{k-2+m}, 2^{k-2+m}]$}\}\\
&J_{k,L}^{m,h}=\{\text{$\exists$ a $p_k(j)$-open horizontal crossing of  $[-2^{k+m},-2^{k-2+m}] \times [-2^{k-2+m}, 2^{k-2+m}]$}\}\\
&J_{k,R}^{m,v}=\{\text{$\exists$ a $p_k(j)$-open vertical crossing of  $[2^{k-1+m},2^{k+m}] \times [-2^{k-1+m}, 2^{k-1+m}]$}\}\\
&J_{k,L}^{m,v}=\{\text{$\exists$ a $p_k(j)$-open vertical crossing of  $[-2^{k+m}, -2^{k-1+m}] \times [-2^{k-1+m}, 2^{k-1+m}]$}\}.\\
\end{split}
\end{equation*}
By \eqref{2.8 of Jarai} and \eqref{D}, (See \cite[Eqs.~(2.19), (2.20)] {MR1981994}),
\begin{equation*}
\begin{split} \label{eq: $J_k(j)$ and $J_{k,R}^m etc}
& \PP(J_k(j)^c)   \leq 16C_3 \exp \left\{-\frac{1}{4}C_4 C_5 \log^{(j)} k \right\} \;\;\text{ and }\\
& \PP(J_{k,R}^m(j)^c \cup J_{k,L}^m(j)^c) \leq 4C_3 \exp \left\{-\frac{1}{2}C_4 C_5 2^m \log^{(j)} k \right\}.
\end{split}
\end{equation*}
By these inequalities, one gets
\begin{align*}  
\PP(H_k^e(j)^c) &\leq \PP \left( J_k(j) \right)^c +\dsum_{m=0}^\infty \PP \left( (J_{k,R}^m(j)^c \cup J_{k,L}^m(j)^c \right)  \\
&\leq (16C_3+C_{6})\exp \left\{-\frac{1}{4}C_4 C_5 \log^{(j)}k \right\}.
\end{align*}
We write $C_{7}$ as $16C_3+C_{6}$ and $c_1$ as $\frac{C_4 C_5}{4}$ for short. Then,
\begin{equation} 
\PP(H_k^e(j)^c) \leq C_{7} \exp\{ - c_1 \log^{(j)}k \} \label{prob(H_k)}.
\end{equation}
 The constant $c_1$ can be made large by choosing $C_5$ large. 
 
 To estimate the mean of $\mathcal{Y}_{A_k}$, we decompose
\begin{equation}
\EE\mathcal{Y}_{A_k}=\EE\left[\mathcal{Y}_{A_k} ; H_k(0)^c \right] + \bigg( \dsum_{j=1}^{\log^*k} \EE\left[\mathcal{Y}_{A_k} ; H_k(j-1) \cap H_k(j)^c \right] \bigg) +  \EE\left[\mathcal{Y}_{A_k} ; H_k(\log^*k)\right].  \label{split of H_k(j)}
\end{equation}
By (\ref{prob(H_k)}) and independence,
\begin{equation} \label{first of A_k}
\begin{split}
 \EE \left[ \mathcal{Y}_{A_k} ; H_k(0)^c \right] & \leq \EE \left[ \mathcal{Y}_{A_k} ; H_k^e(0)^c \right] \leq  \dsum_{e \in A_k} \PP(\omega(e) \in (p_c , p_c +\epsilon], H_k^e(0)^c)\\
&=|A_k| \PP(\omega(e) \in (p_c , p_c +\epsilon]) \PP(H_k^e(0)^c) \\
&\leq |A_k|\epsilon C_{8} e^{-c_1 k}.
\end{split}
\end{equation}
Next, since $\omega(e)$ is independent of $H_k^e(j) \cap \{e \xleftrightarrow{p_k(j)} \infty \}$,
\begin{equation*}  \label{eq: second term of $E[y_{A_k}]$}
\begin{split}
 \EE|\mathcal{Y}_{A_k} ; H_k(j-1) \cap H_k(j)^c|  &=  \dsum_{e \in A_k} \PP( \omega(e) \in (p_c , p_c +\epsilon]  \cap \{ e \xleftrightarrow{p_k(j-1)} \infty \}  \cap H_k^e(j)^c)\\
 &= \epsilon \dsum_{e \in A_k} \PP(e \xleftrightarrow{p_k(j-1)} \infty, H_k^e(j)^c).
\end{split}
\end{equation*}
Applying the FKG inequality and \eqref{prob(H_k)} to this, we obtain
\begin{equation}  \label{second of A_k}
\begin{split}
\EE[\mathcal{Y}_{A_k} ; H_k(j-1) \cap H_k(j)^c] \leq |A_k| \epsilon  \theta(p_k(j-1)) \; C_{7} \exp\{- c_1 \log^{(j)}k \}.\\
\end{split}
\end{equation}
The third term of (\ref{split of H_k(j)}) is bounded above by 
\begin{equation}  \label{third of A_k}
 |A_k|\epsilon \theta(p_k(\log^*k)).
\end{equation}
Using \eqref{bound pi for bigger than critical},  \eqref{bounded D_1} and \eqref{D},
\begin{align*}
\theta(p_k(j)) \leq \frac{ \pi(2^k)}{D_1} (DC_5 \log^{(j)}k)^{1/2}. \label{eq: upperbound_theta p_k(j)}
\end{align*}
Applying this inequality after placing \eqref{first of A_k}, \eqref{second of A_k}, and \eqref{third of A_k} into \eqref{split of H_k(j)}, we obtain
\begin{equation*} \label{eq: upperbound of the number of invaded edges in box}  
\begin{split}
\EE\mathcal{Y}_{A_k}\leq C_{9} |A_k| \epsilon \pi(2^k) \left[ \frac{\exp\{-c_{1} k\}}{\pi(2^k)} + \left\{\dsum_{j=1}^{\log^*k} (\log^{j-1}k)^{1/2-c_{1} } \right\} +1 \right].
\end{split}
\end{equation*}
Since $\pi(2^k) \geq C_{10}2^{-k/2}$ from \eqref{bounded D_1}, we can choose $C_5$ (and therefore $c_1$) so large that 
\begin{align}
\frac{\exp\{-c_{1} k\}}{\pi(2^k)} + \left\{\dsum_{j=1}^{\log^*k} (\log^{j-1}k)^{1/2-c_{1} } \right\} +1 \;\;\text{ is bounded in $k$,} \nonumber
\end{align}
and so
$\EE\mathcal{Y}_{A_k} \leq C_{11} \epsilon 2^{2k} \pi(2^k).$
Recalling $n= 2^K$, we obtain from this and \eqref{bounded D_1} that
\begin{equation*}  
\begin{split}
\EE\mathcal{Y}_n(\epsilon) =  \dsum_{k=1}^K \EE\mathcal{Y}_{A_k} \leq C_{11}   \epsilon 2^{2K} \pi(2^K) \dsum_{k=1}^K \frac{    2^{2k} \pi(2^k)}{   2^{2K} \pi(2^K)} & \leq \frac{C_{11}}{D_1}   \epsilon 2^{2K} \pi(2^K) \dsum_{k=1}^K 2^{2(k-K)} 2^{-\frac{1}{2}(k-K)}\\
& \leq C_{12} \epsilon n^2 \pi(n),
\end{split}
\end{equation*}
completing the proof of the upper bound when $t=1$. The extension to larger $t$ uses the same ideas as in \cite{MR1981994} and \cite[Sec.~3]{MR859839}, so we omit it.


We now turn to the lower bound. For $k \geq 1$, $\epsilon>0$, and any $e \subset A_k$, we let $L_k(e)$ be the event that the following hold:

\begin{enumerate}[(a)]
\item There exists a $p_c$-open circuit $\mathcal{D}$ around the origin in $A_{k-2}$.
\item There exists a $(p_c+\epsilon)$-closed dual circuit around the origin in $A_{k+2}$.
\item  $\mathcal{D}$ is connected to the edge $e \in A_k$ by a $p_c$-open path in $B(2^k)$.
\end{enumerate}

\begin{figure}[!h] 
\centering
\includegraphics[width=90mm]{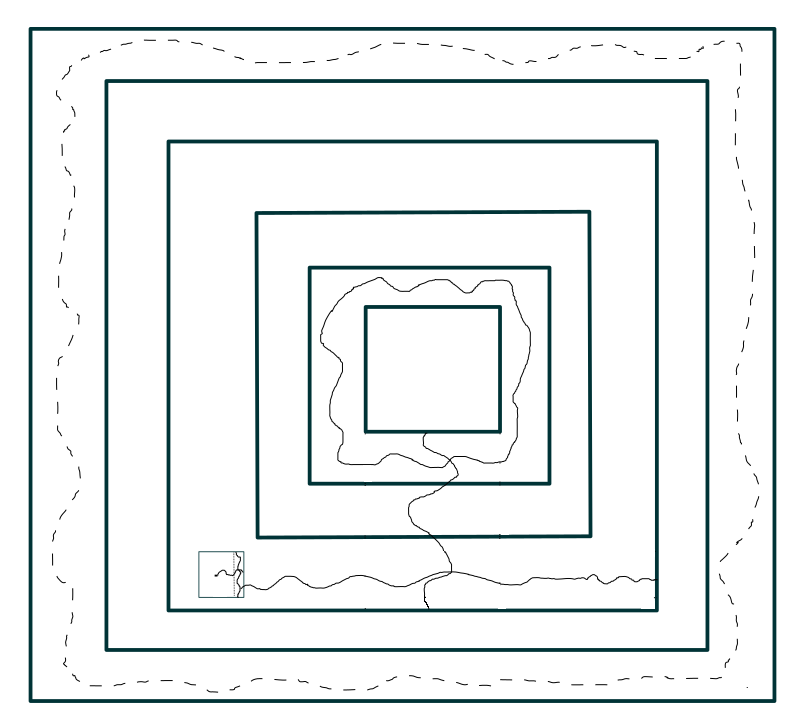}
\caption{The event $L_k(e)$. The boxes, in order from smallest to largest, are $B(2^{k-2})$, $B(2^{k-1})$, $B(2^k)$, $B(2^{k+1})$, $B(2^{k+2})$, and $B(2^{k+3})$. The solid curves are $p_c$-open and the dotted curve is a $(p_c+\epsilon)$-closed dual circuit.}\label{lowerboundofY}
\end{figure}
(See Figure~\ref{lowerboundofY} for an illustration of $L_k(e)$).

If the events described in (a) and (b) both occur, each $(p_c+\epsilon)$-open edge connected to $\mathcal{D}$ by a $(p_c+\epsilon)$-open path will eventually be invaded. Since the event in (b) depends on edge-variables for edges outside of $B(2^{k+1})$, (b) is independent of both (a) and (c). In addition, the events (a) and (c) are increasing. So, by the FKG inequality and the RSW theorem,
\[
\PP(L_k(e)) \geq \PP( (a))  \times \PP( (b)) \times \PP( (c))  \geq C_{13} \PP((b)) \times \PP( (c)).
\]
By a gluing argument \cite[Ch.~11]{grimmett_percolation} using the FKG inequality and the RSW theorem, $\PP( (c) ) \geq C_{14} \pi(2^k)$. Furthermore, as long as $\epsilon$ is so small that $p_c + \epsilon < p_{2^{k+2}}$, then the RSW theorem implies that $\PP((b)) \geq C_{15}$. This means that for such $\epsilon$, one has $\PP(L_k(e)) \geq C_{13}C_{14}C_{15}\pi(2^k)$. Since $\omega(e)$ and the event $L_k(e)$ are independent,
\begin{equation*} 
\begin{split}
\EE|\mathcal{Y}_{A_k}| = \dsum_{e \in A_k} \PP(e \in \text{IPC},  \omega(e) \in (p_c, p_c+\epsilon]) &\geq \dsum_{e \in A_k} \PP( L_k(e),  \omega(e) \in (p_c, p_c+\epsilon]) \\
& \geq C_{16} \epsilon 2^{2k} \pi(2^{k}).
\end{split}
\end{equation*}
For a given $n \geq 1$, choose $k = \lfloor \log_2 n \rfloor$ to complete the proof:
 \[
\EE \mathcal{Y}_n(\epsilon) \geq  \EE \mathcal{Y}_{A_k} \geq C_{16} \epsilon 2^{2k} \pi(2^k) \geq C_{17}\epsilon n^2 \pi(n).
\]
\end{proof}

\subsection{Upper bound} \label{sec: upper bound}
In this section, we show that
\begin{equation}
\limsup_{n \to \infty} a_n(p_c) < 1. \label{eq: to_show_upper_bound}
\end{equation}

To prove \eqref{eq: to_show_upper_bound}, we define
%
%
\[
\Xi_n(\epsilon) = \left[ \tilde{P}_n(p_c+\epsilon) - \tilde{P}_n(p_c) \right] - \left[ \tilde{Q}_n(p_c+\epsilon) - \tilde{Q}_n(p_c)\right],
\]
as the number of edges with weight in the interval $(p_c,p_c+\epsilon]$ which the invasion observes until time $n$ but does not invade, and we give the following proposition.
\begin{proposition} \label{upper theorem 1} There exists $\mathcal{C}_6 >0$ and a function $G$ on $[0, \infty)$ with $\inf_{r \in [0,m]} G(r) > 0$ for each $m \geq 0$ such that for any $C \geq \mathcal{C}_6$, any $n \geq 1$, and any $\epsilon>0$,
\[
\EE \Xi_{\lfloor C n^2 \pi(n) \rfloor}(\epsilon) \geq G(C) \epsilon n^2 \pi(n).
\]
 \end{proposition}

Assuming Proposition~\ref{upper theorem 1} for the moment, let $C \geq \mathcal{C}_6$, and use $\mathbb{E}L_n \leq 4n$ for
\begin{align}
a_{\floor{ C n^2\pi(n) } }(p_c) = \dlim_{\epsilon \downarrow 0} \frac{ \EE \left[  
\tilde{Q}_{ \floor{ C n^2 \pi(n) }} (p_c+\epsilon)- \tilde{Q}_{ \floor{ C n^2 \pi(n) } } (p_c) \right]}{\EE \left[\tilde{P}_{ \floor{ C n^2 \pi(n) } } (p_c+\epsilon) - \tilde{P}_{ \floor{ C n^2 \pi(n)} } (p_c)  \right]} &= \dlim_{\epsilon \downarrow 0} \left( 1 - \frac{ \EE \Xi_{\lfloor Cn^2\pi(n)\rfloor}(\epsilon)}{ \EE L_{\lfloor C n^2 \pi(n) \rfloor}} \right) \nonumber \\
&\leq \dlim_{\epsilon \downarrow 0}  \bigg( 1- \frac{ G(C) \epsilon n^2 \pi( n) }{ 4 C \epsilon  n^2 \pi(n)} \bigg) \nonumber \\
&= 1- \frac{G(C) }{4 C }. \label{eq: G_lower_bound}
\end{align}
Now note that any $n \geq \mathcal{C}_6$ can be written in the form $\lfloor C h^2 \pi(h) \rfloor$ for some integer $h \geq 1$ and some $C \in [\mathcal{C}_6, 4\mathcal{C}_6]$. To see why, observe that any $n \geq \mathcal{C}_6$ is in some interval of the form $[\mathcal{C}_6 h^2 \pi(h), \mathcal{C}_6 (h+1)^2 \pi(h+1))$ for some $h \geq 1$ (since $h^2 \pi(h) \to \infty$ as $h \to \infty$ by \eqref{bounded D_1}). Then because
\[
\frac{\mathcal{C}_6 (h+1)^2 \pi(h+1)}{\mathcal{C}_6 h^2 \pi(h)} = \left( 1+ \frac{1}{h}\right)^2 \frac{\pi(h+1)}{\pi(h)} \leq 4,
\]
we see that $n = \lfloor C_\ast \mathcal{C}_6 h^2 \pi(h) \rfloor$ for some $C_\ast \in [1,4]$. By \eqref{eq: G_lower_bound}, then, we obtain
\[
a_n(p_c) \leq 1- \frac{\inf_{r \in [\mathcal{C}_6,4\mathcal{C}_6]} G(r)}{4 \mathcal{C}_6},
\]
and this implies \eqref{eq: to_show_upper_bound}.


In the remainder of this section, we prove Proposition~\ref{upper theorem 1}. 
%
%
%
\begin{proof}[Proof of Proposition~\ref{upper theorem 1}]
For notational convenience, let $t_n = \lfloor C n^2 \pi(n) \rfloor$. To prove a lower bound on $\Xi_{t_n}(\epsilon)$, we will construct a large $p_c$-open cluster such that with positive probability, independent of $n$, the invasion has intersected this cluster at time $t_n$ and has explored a positive fraction of its boundary edges, but has not yet absorbed the entire cluster. These explored boundary edges will have probability of order $\epsilon$ to have weight in the interval $(p_c,p_c+\epsilon]$, so our lower bound on $\EE \Xi_{t_n}(\epsilon)$ will be of order $\epsilon$ times the size of this explored boundary, which will itself be of order $n^2\pi(n)$.

To construct this cluster, we need several definitions.

\begin{definition} \label{definiton: D(n)}
Define the event $D(n)$ that the following conditions hold:
\begin{enumerate}
\item There exists a $q_n$-open circuit around the origin in Ann($n,2n$).
\item There exists an edge $f \in$ Ann$(6n, 7n)$ with $\omega(f) \in (q_{n}, p_c)$ such that:
\begin{enumerate}
\item there exists a $p_c$-closed dual path $P$ around the origin in Ann($4n,8n)^*  \backslash \{f^*\}$ that is connected to the endpoints of $f^*$ so that $P \cup \{f^*\}$ is a dual circuit around the origin, and
\item there exists a $p_c$-open path connecting an endpoint of $f$ to $B(n)$, and another disjoint $p_c$-open path connecting the other endpoint of $f$ to $\partial B(16n)$.
\end{enumerate}
\item There exists a $p_c$-open circuit around the origin in Ann($8n, 16n)$.
\end{enumerate}
For $e \subset Ann(2n,4n)$, define $D^e(n)$ as the event that $D(n)$ occurs without using the edge $e$. (That is, $D(n)$ occurs and the first connection listed in 2(b) does not use $e$.)
\end{definition}
See Figure~\ref{fig: D(n)} for an illustration of $D(n)$.

\begin{figure}[h] 
\centering
\includegraphics[width=85mm]{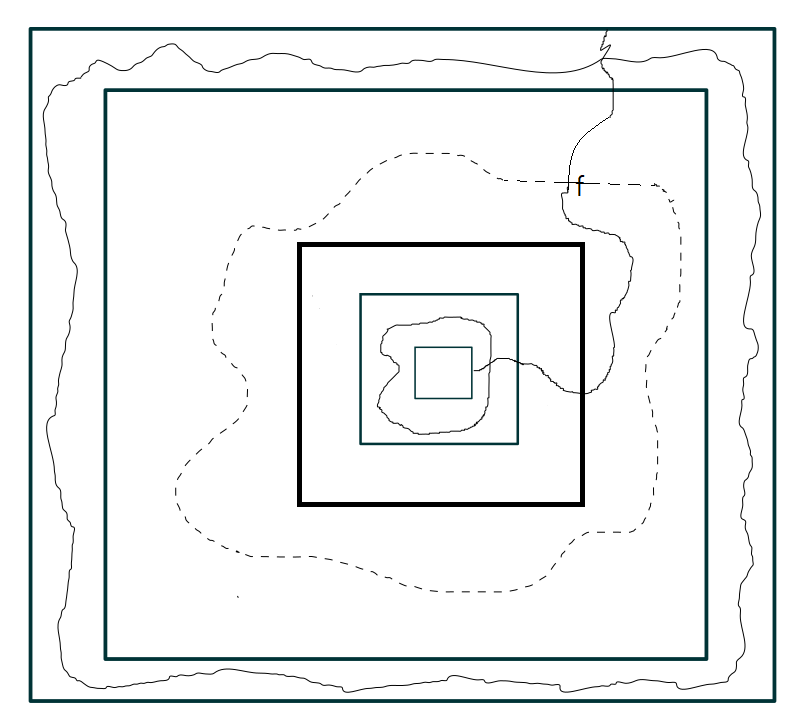}
\caption{The event $D(n)$. The boxes, in order from smallest to largest, are $B(n)$, $B(2n)$, $B(4 n)$, $B(8n)$ and $B(16n)$. The solid circuit in Ann($n,2n$) is $q_n$-open and the path from $\partial B(n)$ to $f$ is $p_c$-open; the dotted dual path in Ann($4n, 8n$) is $p_c$-closed, $\omega(f) \in (q_{n}, p_c)$, and the other solid paths are $p_c$-open.}
\label{fig: D(n)}
\end{figure}

When the event $D(n)$ occurs, we can define $\mathcal{C}_\ast$ as the innermost $q_n$-open circuit around the origin in $Ann(n,2n)$ and $D_\ast$ as the outermost $p_c$-open circuit around the origin in $Ann(8n,16n)$. Note that on $D(n)$, the circuits $\mathcal{C}_\ast$ and $D_\ast$ are part of the same $p_c$-open cluster; this will form part of our ``large cluster'' referenced above. We need to make sure that we have started to invade this cluster, but are not yet done at time $t_n$, so we define stopping times
\begin{align*}
t_{D_\ast} &= \text{ first time at which the invasion invades an edge from $D_\ast$} \\
T_{D_\ast} &= \text{ first time at which the invasion invades the entire $p_c$-open cluster of $D_\ast$}.
\end{align*}
Note that on $D(n)$, we have $t_{D_\ast} \leq T_{D_\ast}$ and trivially,
\begin{equation}\label{eq: step_1}
\EE \Xi_{t_n}(\epsilon) \geq \EE \Xi_{t_n}(\epsilon) \mathbf{1}_{D(n) \cap \{t_{D_\ast} \leq t_n < T_{D_\ast}\}}.
\end{equation}
The next lemma shows that on the events listed on the right, $\Xi_{t_n}(\epsilon)$ is, on average, at least order $\epsilon$ times the cardinality of a certain subset of the edge boundary of the $p_c$-open cluster of $D_\ast$. For this we define the size $Y_n$ of this subset:
\[
Y_n =\#\{e \subset Ann(2n,4n) : \omega(e) > p_c, e \xleftrightarrow{q_{n}} \partial B(n) \text{ in }B(4n)\}.
\]
\begin{lemma}\label{lem: percolation_cluster}
For any $n \geq 1$,
\[
\EE \Xi_{t_n}(\epsilon) \mathbf{1}_{D(n) \cap \{t_{D_\ast} \leq t_n < T_{D_\ast}\}} \geq \frac{\epsilon}{1-p_c} \EE Y_n \mathbf{1}_{D(n) \cap \{t_{D_\ast} \leq t_n < T_{D_\ast}\}}.
\]
\end{lemma}
\begin{proof}
First we let
\[
\hat{Y}_n = \#\{e \subset Ann(2n,4n) : \omega(e) \in (p_c,p_c+\epsilon], e \xleftrightarrow{q_{n}} \partial B(n) \text{ in }B(4n)\}.
\]
On the event $D(n) \cap \{t_{D_\ast} \leq t_n < T_{D_\ast}\}$, any edge in the set which defines $\hat{Y}_n$ will be observed by the invasion until time $t_n$ but will not be invaded (that is, it is counted in the definition of $\Xi_n(\epsilon)$). To see why, let $e$ be an edge in the set which defines $\hat{Y}_n$. First, we must show that $e$ is not invaded at time $t_n$. This is because, in order for the invasion to even observe $e$, it must first pass through the circuit $\mathcal{C}_\ast$. Since $\omega(e) > p_c$, the invasion will invade the entire $p_c$-open cluster of $\mathcal{C}_\ast$ (which equals the $p_c$-open cluster of $D_\ast$) before it invades $e$. Since $t_n < T_{D_\ast}$, $e$ cannot be invaded at time $t_n$. Second, we must show that $e$ is observed by time $t_n$. The reason is that since $t_{D_\ast} \leq t_n$, at time $t_n$, the invasion has already invaded an edge from $D_\ast$. Since $D(n)$ occurs, the edge $f$ must therefore be invaded before time $t_{D_\ast}\leq t_n$. Before $f$ can be invaded, the entire $q_n$-open cluster of $\mathcal{C}_\ast$ must be invaded, so at least one endpoint of $e$ is in the invasion at time $t_n$. This means that $e$ is observed by time $t_n$. In conclusion,
\begin{align*}
&\EE \Xi_{t_n}(\epsilon) \mathbf{1}_{D(n) \cap \{t_{D_\ast} \leq t_n < T_{D_\ast}\}} \\
\geq~& \EE \hat{Y}_n \mathbf{1}_{D(n) \cap \{t_{D_\ast} \leq t_n < T_{D_\ast}\}} \\
=~& \sum_{e \subset \text{Ann}(2n,4n)} \mathbb{P}(\omega(e) \in (p_c,p_c+\epsilon], e \xleftrightarrow{q_n} \partial B(n) \text{ in } B(4n), D(n), t_{D_\ast} \leq t_n < T_{D_\ast}).
\end{align*}

The second and final step is to show that for all $e \subset \text{Ann}(2n,4n)$, we have
\begin{equation}\label{eq: perc_cluster}
\begin{split}
&\mathbb{P}(\omega(e) \in (p_c,p_c+\epsilon], e \xleftrightarrow{q_n} \partial B(n) \text{ in } B(4n), D(n), t_{D_\ast} \leq t_n < T_{D_\ast}) \\
=~& \frac{\epsilon}{1-p_c} \mathbb{P}(\omega(e) >p_c, e \xleftrightarrow{q_n} \partial B(n) \text{ in } B(4n), D(n), t_{D_\ast} \leq t_n < T_{D_\ast}).
\end{split}
\end{equation}
Once this is done, we can sum the right side and obtain the statement of the lemma.

To argue for \eqref{eq: perc_cluster}, we need to be able to decouple the value of $\omega(e)$ from the other events. Intuitively this should be possible because when $D(n)$ occurs, after the invasion touches $\mathcal{C}_\ast$, it does not need to check any weights for edges which are $p_c$-closed until after time $T_{D_\ast}$. To formally prove this, we represent the weights $(\omega(e))$ used for the invasion as functions of three independent variables. This representation is used in the ``percolation cluster method'' of Chayes-Chayes-Newman, but their method uses them in a dynamic way, whereas ours will be static. For this representation, we assign different variables to the edges: let $(U_e^1,U_e^2,\eta_e)_{e \in \mathcal{E}^2}$ be an i.i.d. family of independent variables, where $U_e^1$ is uniform on $[0,p_c]$, $U_e^2$ is uniform on $(p_c,1]$, and $\eta_e$ is Bernoulli with parameter $p_c$. Then we set 
\[
\omega(e) = \begin{cases}
U_e^1 & \quad \text{if } \eta_e = 1 \\
U_e^2 & \quad \text{if } \eta_e = 0.
\end{cases}
\]
Next, we define another invasion percolation process $(\hat{G}_n)$ (a sequence of growing subgraphs) as follows. If $D(n)$ does not occur, then $\hat{G}_n$ is equal to $(0, \{\})$ for all $n$ (it stays at the origin with no edges). If $D(n)$ does occur, then $\hat{G}_n$ proceeds according to the usual invasion rules (with the weights $(\omega(e))$) until it reaches $\mathcal{C}_\ast$. After it contains a vertex of $\mathcal{C}_\ast$, it no longer checks the $\omega$-value of any edge $\hat{e}$ with $\eta_{\hat{e}} = 0$ (it only checks the $\eta$-value). When there are no more edges with $\eta$-value equal to one for the invasion to invade, it stops (we set $\hat{G}_n$ to be constant after this time). Associated to this new invasion will be stopping times similar to $t_{D_\ast}$ and $T_{D_\ast}$:
\begin{align*}
\hat{t}_{D_\ast} &= \text{ first time at which the new invasion invades an edge from $D_\ast$} \\
\hat{T}_{D_\ast} &= \text{ first time at which the new invasion invades the entire $p_c$-open cluster of $D_\ast$}.
\end{align*}
Note that if $D(n)$ does not occur, $\hat{t}_{D_\ast} = \hat{T}_{D_\ast} = \infty$, and that if $D(n)$ occurs, $\hat{T}_{D_\ast}$ equals the first time after which the graphs $\hat{G}_n$ become constant.

Given these definitions, the top equation of \eqref{eq: perc_cluster} equals
\begin{equation*}\label{eq: new_definitions}
\mathbb{P}(U_e^2 \in (p_c,p_c+\epsilon], \eta_e=0, e \xleftrightarrow{q_n} \partial B(n) \text{ in } B(4n), D(n), t_{D_\ast} \leq t_n < T_{D_\ast}).
\end{equation*}
We then claim that
\begin{equation}\label{eq: replace_invasion}
\begin{split}
&\mathbb{P}(U_e^2 \in (p_c,p_c+\epsilon], \eta_e=0, e \xleftrightarrow{q_n} \partial B(n) \text{ in } B(4n), D(n), t_{D_\ast} \leq t_n < T_{D_\ast}) \\
=~&\mathbb{P}(U_e^2 \in (p_c,p_c+\epsilon], \eta_e=0, e \xleftrightarrow{q_n} \partial B(n) \text{ in } B(4n), D(n), \hat{t}_{D_\ast} \leq t_n < \hat{T}_{D_\ast}).
\end{split}
\end{equation}
This equation holds because when $D(n)$ occurs, $t_{D_\ast} = \hat{t}_{D_\ast}$ and $T_{D_\ast} = \hat{T}_{D_\ast}$. Indeed, if $D(n)$ occurs, then both invasions $(G_n)$ and $(\hat{G}_n)$ are equal until they touch $\mathcal{C}_\ast$. After this time, the original invasion $(G_n)$ does not invade any $p_c$-closed edges until time $T_{D_\ast}$, and neither does $(\hat{G}_n)$ (by definition). This shows \eqref{eq: replace_invasion}.

Now that we have \eqref{eq: replace_invasion}, we simply note that because $(\hat{G}_n)$ does not use any edges in $B(2n)^c$ that are $p_c$-closed, the the times $\hat{t}_{D_\ast}$ and $\hat{T}_{D_\ast}$ are independent of $(U_e^2)_{e \in B(2n)^c}$. Furthermore, the events $\{\eta_e=0\}$, $\{e \xleftrightarrow{q_n} \partial B(n)\text{ in } B(4n)\}$, and $D(n)$ are independent of $(U_e^2)_{e \in B(2n)^c}$, and $U_e^2 \in (p_c,p_c+\epsilon]$ depends only on $(U_e^2)_{e \in B(2n)^c}$. By independence, therefore, the lower equation of \eqref{eq: replace_invasion} is equal to
\[
\frac{\epsilon}{1-p_c} \mathbb{P}(\eta_e=0, e \xleftrightarrow{q_n} \partial B(n) \text{ in } B(4n), D(n), \hat{t}_{D_\ast} \leq t_n < \hat{T}_{D_\ast}),
\]
which equals the bottom equation in \eqref{eq: perc_cluster}. This shows \eqref{eq: perc_cluster}.
\end{proof}

Combining Lemma~\ref{lem: percolation_cluster} with \eqref{eq: step_1}, and then reducing to the subevent $D^e(n)$ (recall this is the subevent of $D(n)$ on which the paths involved in $D(n)$ do not use the given $e \subset$ Ann$(2n,4n)$),  we obtain
\begin{align}
&\EE \Xi_{t_n} \nonumber \\
\geq~& \frac{\epsilon}{1-p_c} \EE Y_n \mathbf{1}_{D(n) \cap \{t_{D_\ast} \leq t_n < T_{D_\ast}\}} \nonumber \\
\geq~& \frac{\epsilon}{1-p_c} \sum_{e \subset \text{Ann}(2n,4n)} \mathbb{P}(\omega(e) > p_c, e \xleftrightarrow{q_n} \partial B(n) \text{ in } B(4n), D^e(n), t_{D_\ast} \leq t_n < T_{D_\ast}) \label{eq: step_2a}.
\end{align}
The most difficult part of the above sum is the term $t_n < T_{D_\ast}$. To ensure that this occurs, we will construct a large set of vertices in the exterior of $D_\ast$ which will connect to $D_\ast$ by $p_c$-open paths. To do this, we will need to use independence to separate the interior of $D_\ast$ from its exterior, using the following two events, which comprise pieces of the event $D(n)$.

\begin{definition} \label{def:D_int^e} For any circuit $\hat{D}_* \subset \text{Ann}(8n,16n)$ around the origin, define the event $D^e_{int}(n, \hat{D}_*)$ that the following hold.
\begin{enumerate}
\item There exists a $q_n$-open circuit around the origin in Ann($n,2n$).
\item There exists an edge $f \subset$ Ann$(6n, 7n)$ with $\omega(f) \in (q_{n}, p_c)$ such that:
\begin{enumerate}
\item there exists a $p_c$-closed dual path $P$ around the origin in Ann($4n,8n)^*  \backslash \{f^*\}$ that is connected to the endpoints of $f^*$ so that $P \cup \{f^*\}$ is a circuit around the origin, and
\item there exists a $p_c$-open path connecting an endpoint of $f$ to $B(n)$ (avoiding $e$), and another disjoint $p_c$-open path connecting the other endpoint of $f$ to $\hat{D}_\ast$.
\end{enumerate}
\end{enumerate}
We also define the event $D_{ext}(n, \hat{D}_*)$ that the following hold.
\begin{enumerate}
\item There exists a $p_c$-open path from $\hat{D}_*$ to $\partial B(16n)$.
\item $\hat{D}_*$ is the outermost $p_c$-open circuit in Ann($8n,16n)$.
\end{enumerate}
\end{definition}

Directly from the definitions, we note that for any circuit $\hat{D}_* \subset$ Ann($8n,16n)$, $D^e_{int}(n ,\hat{D}_*) \cap D_{ext}(n, \hat{D}_*)$ implies $D^e(n)$ (actually the union over $\hat{D}_\ast$ of this intersection is equal to $D^e(n)$), and the events $D^e_{int}(n, \hat{D}_*)$ and $D_{ext}(n , \hat{D}_*)$ are independent. Last, for distinct $\hat{D}_\ast$, the events $(D_{int}^e(n,\hat{D}_\ast) \cap D_{ext}(n,\hat{D}_\ast))_{\hat{D}_\ast}$ are disjoint. Decomposing \eqref{eq: step_2a} over the choice of the outermost circuit $\hat{D}_\ast$, we obtain that $\EE \Xi_{t_n}(\epsilon)$ equals
\[
\frac{\epsilon}{1-p_c} \sum_{e \subset \text{Ann}(2n,4n)} \sum_{\hat{D}_\ast} \mathbb{P}\left( \begin{array}{c}
\omega(e) > p_c, e \xleftrightarrow{q_n} \partial B(n) \text{ in } B(4n), D^e_{int}(n,\hat{D}_\ast), \\
 D_{ext}(n,\hat{D}_\ast), t_{\hat{D}_\ast} \leq t_n < T_{\hat{D}_\ast}
 \end{array} \right).
\]
(Here $t_{\hat{D}_\ast}$ and $T_{\hat{D}_\ast}$ are similar to $t_{D_\ast}$ and $T_{D_\ast}$ but defined for the detministic circuit $\hat{D}_\ast$.) Note that $\{t_{\hat{D}_\ast} \leq t_n\}$ depends only on the weights in the interior of $\hat{D}_\ast$, but $\{t_n <T_{\hat{D}_\ast}\}$ does not depend only on the exterior. To force this dependence, we simply create a large $p_c$-open cluster in the exterior of $\hat{D}_\ast$. For our deterministic $\hat{D}_\ast$, let
\[
Z(\hat{D}_\ast) = \#\{e \subset B(16)^c : \omega(e) < p_c,~e \xleftrightarrow{p_c} \hat{D}_\ast\}.
\]
If $Z(\hat{D}_\ast) > Cn^2\pi(n)$ on $D^e_{int}(n,\hat{D}_\ast) \cap D_{ext}(n,\hat{D}_\ast)$, then $t_n < T_{\hat{D}_\ast}$. Since this event depends on variables for edges in the exterior of $\hat{D}_\ast$, we can use independence for the lower bound for $\EE \Xi_{t_n}(\epsilon)$ of
\begin{equation}\label{eq: step_2}
\begin{split}\frac{\epsilon}{1-p_c} \sum_{e \subset \text{Ann}(2n,4n)} \sum_{\hat{D}_\ast} \bigg[ &\mathbb{P}\left( 
\omega(e) > p_c, e \xleftrightarrow{q_n} \partial B(n) \text{ in } B(4n), D^e_{int}(n,\hat{D}_\ast), t_{\hat{D}_\ast} \leq t_n\right) \\
&\times \mathbb{P}\left( D_{ext}(n,\hat{D}_\ast), Z(\hat{D}_\ast) > Cn^2 \pi(n) \right) \bigg].
\end{split}
\end{equation}
Note that only the first factor inside the double sum depends on $e$. To bound it, we give the next lemma. 
\begin{lemma}\label{lem: D(n)_bound}
There exists $\mathcal{C}_6$ and $C_{18}>0$ such that for all $n \geq 1$, all $\hat{D}_\ast$ around the origin in Ann$(8n,16n)$, and all $C\geq \mathcal{C}_6$,
\[
\sum_{e \subset \text{Ann}(2n,4n)} \mathbb{P}\left( \omega(e) > p_c, e \xleftrightarrow{q_n} \partial B(n) \text{ in } B(4n), D^e_{int}(n,\hat{D}_\ast), t_{\hat{D}_\ast} \leq t_n\right) \geq C_{18} n^2 \pi(n).
\]
\end{lemma}
\begin{proof}
First note that for any $\hat{D}_\ast$, we have $t_{\hat{D}_\ast} \leq t_n$ whenever $R_{t_n} \geq 16n$. Therefore it will suffice to show a lower bound for 
\begin{align*} \label{eq: lowerbound_with_R}
\sum_{e \subset \text{Ann}(2n,4n)} \mathbb{P}\left( \omega(e) > p_c, e \xleftrightarrow{q_n} \partial B(n) \text{ in } B(4n), D^e_{int}(n,\hat{D}_\ast), R_{t_n} \geq 16n \right).
\end{align*}
To do this, we will show both a lower bound
\begin{equation}\label{eq: lower_one}
\sum_{e \subset \text{Ann}(2n,4n)} \mathbb{P}\left( \omega(e) > p_c, e \xleftrightarrow{q_n} \partial B(n) \text{ in } B(4n), D^e_{int}(n,\hat{D}_\ast) \right) \geq C_{19}n^2\pi(n)
\end{equation}
and an upper bound
\begin{equation}\label{eq: upper_one}
\sum_{e \subset \text{Ann}(2n,4n)} \mathbb{P}\left( \omega(e) > p_c, e \xleftrightarrow{q_n} \partial B(n) \text{ in } B(4n), D^e_{int}(n,\hat{D}_\ast), R_{t_n} < 16n \right) \leq \frac{C_{19}}{2} n^2\pi(n),
\end{equation}
for all $n$, so long as $C$ is larger than some $\mathcal{C}_6$.

Inequality \eqref{eq: upper_one} is easier, so we start with it. First sum over $e$ and then apply the Cauchy-Schwarz inequality to get the upper bound
\begin{align*}
&\left( \mathbb{E}\left( \#\{e \subset \text{Ann}(2n,4n) : e \xleftrightarrow{p_c} \partial B(n) \text{ in } B(4n)\} \right)^2 \right)^{1/2} \bigg( \mathbb{P}( R_{t_n} < 16n) \bigg)^{1/2} \\
\leq~&\left( \sum_{e,f \subset \text{Ann}(2n,4n)} \mathbb{P}(e \xleftrightarrow{p_c} \partial B(e,n/2), f \xleftrightarrow{p_c} \partial B(f,n/2)) \right)^{1/2} \bigg( \mathbb{P}( R_{t_n} < 16n) \bigg)^{1/2}.
\end{align*}
Here, for example, $B(f,n/2)$ is the box of sidelength $n$ centered at the bottom-left endpoint of $e$. The fact that the sum is bounded by $(C_{20} n^2 \pi(n))^2$ follows from standard arguments, like those in \cite[p. 388-391]{MR859839}. (See the upper bound for $\EE Z_n(\ell_0)^2$ we give in full detail below \eqref{eq: PZ2} for a nearly identical calculation.) This gives us the bound
\[
\text{LHS of } \eqref{eq: upper_one} \leq C_{20}n^2\pi(n) \sqrt{\mathbb{P}(R_{t_n} < 16n)}.
\]
Due to Lemma~\ref{lemma: upperbound of R}, given any $C_{19}$ from \eqref{eq: lower_one} (assuming we show that inequality, which we will in a moment), we can find $\mathcal{C}_6$ such that for $C\geq \mathcal{C}_6$,
\[
C_{20}\sqrt{\mathbb{P}(R_{t_n} < 16n)} \leq C_{19}/2,
\]
and this completes the proof of \eqref{eq: upper_one}.

Turning to the lower bound \eqref{eq: lower_one},
since $\omega(e)$ is independent of both events $\{e \xleftrightarrow{q_n} \partial B(n) \text{ in } B(4n)\}$ and $D_{int}^e(n, \hat{D}_\ast)$, 
 \begin{align}
& \sum_{e \subset \text{Ann}(2n,4n)} \mathbb{P}\left( \omega(e) > p_c, e \xleftrightarrow{q_n} \partial B(n) \text{ in } B(4n),   D^e_{int}(n,\hat{D}_\ast)  \right) \nonumber \\
=~&(1-p_c) \sum_{e \subset \text{Ann}(2n,4n)} \mathbb{P}\left(e \xleftrightarrow{q_n} \partial B(n) \text{ in } B(4n),  D^e_{int}(n,\hat{D}_\ast)  \right) \label{eq: step_2_aa}.
\end{align} 
Estimating each summand from below uses some standard gluing constructions (see \cite[Thm.~1]{MR879034} or \cite[Lemma~6.3]{MR2573559} for some examples), so we will only indicate the main idea. It will suffice to lower bound the sum over only $e \subset \hat{B}_n := [-4n,-2n]\times [-2n,2n]$. To construct the event $D^e_{int}(n)$, we build the event $\bar{D}(n)$, defined by the following conditions:
\begin{enumerate}[{[a]}]
\item There exists a $q_n$-open circuit around the origin in Ann($n,2n$).
\item[] There exists an edge $f \subset$ $B'(n):=$Ann$(6n,7n) \cap [6n,\infty)^2$ with $\omega(f) \in (q_{n}, p_c)$ such that:
\item there exists a $p_c$-closed dual path $P$ around the origin in Ann($4n,8n)^*  \backslash \{f^*\}$ that is connected to the endpoints of $f^*$ so that $P \cup \{f^*\}$ is a circuit around the origin, and
\item there exists a $p_c$-open path connecting one endpoint of $f$ to $B(n)$ and remaining in $[-n,\infty) \times \mathbb{R}$. Also, there exists another disjoint $p_c$-open path connecting the other endpoint of $f$ to $\partial B(16n)$.
\end{enumerate}

 Figure~\ref{fig: C_f(n)} illustrates the event $\bar{D}(n)$.
\begin{figure}[h] 
\centering
\includegraphics[width=80mm]{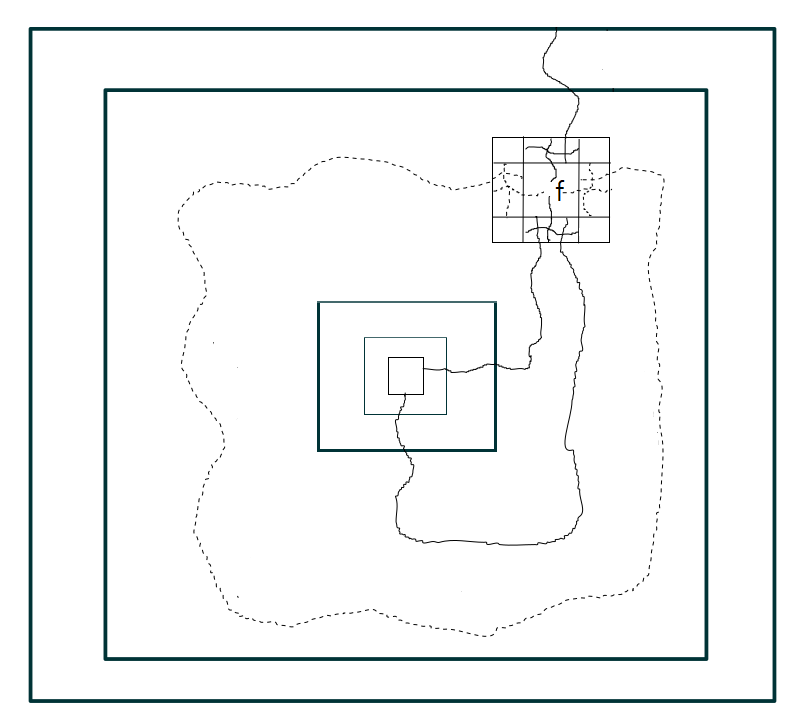} 
\caption{The event $\bar{D}(n)$. The boxes, in order from smallest to largest, are $B(n)$, $B(2n)$, $B(4n)$, $B(8n)$ and $B(16n)$. The solid circuit in Ann($2n,4n$) is $q_n$-open. The solid paths from $\partial B(n)$ to $f$ and the solid path from $f$ to $\partial B(16n)$ are $p_c$-open. The dotted circuit in Ann($4n, 8n$) is $p_c$-closed.} \label{fig: C_f(n)}
\end{figure}

The event described in $[b]$ guarantees item 2(a) in the definition of $D^e_{int}(n,\hat{D}_\ast)$. Since the event described in $[c]$ has a $p_c$-open path from $\partial B(n)$ to $\partial B(16n)$ containing $f$ without using $e$, the event $[c]$ implies item 2(b) in the definition of $D^e_{int}(n,\hat{D}_\ast)$. Therefore, for any circuit $\hat{D}^e_\ast \subset$ Ann($8n,16n$), we can estimate the sum in the bottom of \eqref{eq: step_2_aa}:
\begin{align}
&\sum_{e \subset \text{Ann}(2n,4n)} \mathbb{P}\left(e \xleftrightarrow{q_n} \partial B(n) \text{ in } B(4n), D^e_{int}(n,\hat{D}_\ast)   \right) \nonumber \\
\geq~& \sum_{e \subset \hat{B}_n} \mathbb{P}\left(e \xleftrightarrow{q_n} \partial B(n) \text{ in } B(4n) \cap (-\infty,-n] \times \mathbb{R}, \bar{D}(n) \right) \label{eq: step_2_bb}.
\end{align}

By applying the generalized FKG inequality (positive correlation for certain increasing and decreasing events, so long as they depend on particular regions of space --- see \cite[Lem.~13]{MR2438816}) and a gluing construction, one can decouple the events described in $\bar{D}(n)$ and the event $\{e \xleftrightarrow{p_c} \partial B(n)$ in $B(4n) \cap (-\infty,-n] \times \mathbb{R}\}$ to obtain the lower bound for \eqref{eq: step_2_bb} of
\begin{align}
&\sum_{e \subset \hat{B}_n} \mathbb{P}(e \xleftrightarrow{q_n} \partial B(n) \text{ in } B(4n)\cap (-\infty,-n]\times \mathbb{R}) \mathbb{P}([a]) \mathbb{P}([b],[c]) \nonumber \\
\geq~& c_{2} \mathbb{P}([b],[c]) \sum_{e \subset \hat{B}_n} \mathbb{P}(e \xleftrightarrow{q_n} \partial B(n) \text{ in } B(4n) \cap (-\infty,-n] \times \mathbb{R}). \label{eq: step_2_cc}
\end{align}

To give a lower bound for $\mathbb{P}([b],[c])$, let $A(n,f)$ be the event described in [b] and [c] (along with the condition $\omega(f) \in (q_n,p_c)$), so that this probability equals $\mathbb{P}(\cup_f A(n,f))$, and the union is over $f \subset \text{Ann}(6n,7n) \cap [6n,\infty)^2$. Letting $A'(n,f)$ be the same event, but with the $p_c$-open paths from [c] replaced by $q_n$-open paths, we obtain
\[
\mathbb{P}([b],[c]) = \mathbb{P}(\cup_f A(n,f)) \geq \mathbb{P}(\cup_f A'(n,f)).
\]
Note that the events $A'(n,f)$ for distinct $f$ are disjoint. Therefore
\begin{equation}\label{eq: cheese_pizza_1}
\mathbb{P}([b],[c]) \geq \sum_f \mathbb{P}(A'(n,f)).
\end{equation}
By a gluing argument involving the RSW theorem, the generalized FKG inequality, and Kesten's arms direction method (see \cite[Eq.~(2.9)]{MR879034}), if we define $B(n,f)$ as the event that there are two disjoint $q_n$-open paths connecting $f$ to $\partial B(f,n)$, and two disjoint $p_c$-closed dual paths connecting $f^*$ to $\partial B(f,n)$, then by using independence of the value of $\omega(f)$ from the event $A'(n,f)$, we can obtain
\begin{equation}\label{eq: cheese_pizza_2}
\mathbb{P}(A'(n,f)) \geq c_{3} (p_c-q_n) \mathbb{P}(B(f,n)).
\end{equation}
Last, by a variant of \cite[Lemma~6.3]{MR2573559} (instead of taking $p,q \in [p_c,p_n]$, one takes $p,q \in [q_n,p_c]$, with $p=q_n$ and $q = p_c$, and the proof is nearly identical), we have $\mathbb{P}(B(f,n)) \asymp \mathbb{P}(A_n^{2,2})$, where $A_n^{2,2}$ is the four-arm event from \eqref{definition of $A_n^{2,2}$}. Using this with \eqref{eq: cheese_pizza_1} and \eqref{eq: cheese_pizza_2} gives
\[
\mathbb{P}([b],[c]) \geq c_{4} (p_c-q_n) \sum_f \mathbb{P}(A_n^{2,2}).
\]
By \eqref{definition of $A_n^{2,2}$}, we establish $\mathbb{P}([b],[c]) \geq c_{5}$, and putting this in \eqref{eq: step_2_cc},
\begin{align}
&\sum_{e \subset \hat{B}_n} \mathbb{P}(e \xleftrightarrow{q_n} \partial B(n) \text{ in } B(4n)\cap (-\infty,-n]\times \mathbb{R}) \mathbb{P}([a]) \mathbb{P}([b],[c]) \nonumber \\
\geq~& c_{2} c_{5} \sum_{e \subset \hat{B}_n} \mathbb{P}(e \xleftrightarrow{q_n} \partial B(n) \text{ in } B(4n) \cap (-\infty,-n]\times \mathbb{R}). \label{eq: step_2_dd}
\end{align}

Last, to deal with the summand of \eqref{eq: step_2_dd}, we can use a gluing construction along with the FKG inequality and the RSW theorem to obtain
\[
\mathbb{P}(e \xleftrightarrow{q_n} \partial B(n)) \geq c_{6} \mathbb{P}(e \xleftrightarrow{q_n}\partial B(e,dist(e,\partial B(n))),
\]
where $dist$ is the $\ell_\infty$-distance. By \eqref{compare for different p} and \eqref{bounded D_1}, we have
\[
\mathbb{P}(e \xleftrightarrow{q_n} \partial B(e,dist(e,\partial B(n)))) \geq c_{7} \mathbb{P}(e \xleftrightarrow{p_c} \partial B(e,dist(e,\partial B(n)))) \geq c_{8}\pi(n).
\]
Placing this in \eqref{eq: step_2_dd} and summing over $e$ finally gives
\[
\sum_{e \subset \text{Ann}(2n,4n)} \mathbb{P}\left( \omega(e) > p_c, e \xleftrightarrow{q_n} \partial B(n) \text{ in } B(4n),   D^e_{int}(n,\hat{D}_\ast)  \right) \geq c_{9}n^2 \pi(n),
\]
which finishes the proof of \eqref{eq: lower_one}.
 
\end{proof}

Applying the lemma to the lower bound from \eqref{eq: step_2}, we obtain for all $C \geq \mathcal{C}_6$
\begin{align*}
\EE \Xi_{t_n}(\epsilon) &\geq \frac{\epsilon}{1-p_c} C_{21} n^2\pi(n) \sum_{\hat{D}_\ast} \mathbb{P}(D^e_{ext}(n,\hat{D}_\ast), Z(\hat{D}_\ast) > Cn^2 \pi(n)) \\
&= \frac{\epsilon}{1-p_c}C_{21}n^2\pi(n) \mathbb{P}\left( \bigcup_{\hat{D}_\ast} \{D^e_{ext}(n,\hat{D}_\ast), Z(\hat{D}_\ast) > Cn^2 \pi(n)\}\right) \\
&\geq \frac{\epsilon}{1-p_c} C_{21}n^2\pi(n) \mathbb{P}(A_n,B_n(C)),
\end{align*}
where $A_n$ is the event that there is a $p_c$-open circuit around the origin in Ann$(8n,16n)$ and $B_n(C)$ is the event that there are more than $Cn^2\pi(n)$ vertices in $B(16n)^c$ connected to $B(16n)$ by $p_c$-open paths. By the FKG inequality and the RSW theorem,
\begin{equation}\label{eq: step_3}
\EE \Xi_{t_n}(\epsilon) \geq \frac{\epsilon}{1-p_c}C_{21}C_{22} n^2 \pi(n) \mathbb{P}(B_n(C)) \text{ for $n \geq 1$, all $\epsilon>0$, and $C \geq \mathcal{C}_6$}.
\end{equation}

Last, we argue that there exists a function $F$ on $[0,\infty)$ such that $\inf_{r \in [0,m]}  F(r) > 0$ for each $m \geq 0$ and such that
\begin{equation}\label{eq: step_4_again}
\mathbb{P}(B_n(C)) \geq  F(C) \text{ for all } n \geq 1 \text{ and } C \geq 0.
\end{equation}
Combining this with \eqref{eq: step_3} and setting $G(C) = C_{21} C_{22} F(C) / (1-p_c)$ will complete the proof of Proposition~\ref{upper theorem 1} and therefore of the proof of the upper bound in Theorem~\ref{thm: main theorem}.

To show \eqref{eq: step_4_again}, we use some standard percolation arguments. For $\ell \geq 5$, set
\[
Z_n(\ell) := \#\{v \in \text{Ann}(2^\ell n, 2^{\ell+1} n) : v \xleftrightarrow{ \; p_c \;} \partial B(16n) \}.
\]
By definition of $Z_n(\ell)$ and $B_n(C)$, 
\begin{equation}\label{eq: implication}
\mathbb{P}(B_n(C)) \geq \mathbb{P}(Z_n(\ell) > Cn^2\pi(n)) \text{ for any }\ell \geq 5.
 \end{equation}
 To give a lower bound for the probability of $Z_n(\ell)$, we use the second moment method (Paley-Zygmund inequality):
\begin{equation}\label{eq: PZ}
\mathbb{P}\left(Z_n(\ell) \geq \frac{1}{2} \EE Z_n(\ell)\right) \geq \frac{1}{4} \frac{(\EE Z_n(\ell))^2}{\EE Z_n(\ell)^2}.
\end{equation}
Accordingly, we need a lower bound for $\EE Z_n(\ell)$ and an upper bound for $\EE Z_n(\ell)^2$. 

To bound $\EE Z_n(\ell)$ from below, note that if there is a $p_c$-open circuit around the origin in $\text{Ann}(2^{\ell+1}n,2^{\ell+2}n)$ and a $p_c$-open path connecting $B(16n)$ to $\partial B(2^{\ell+2}n)$, then any point $v \in \text{Ann}(2^\ell n, 2^{\ell+1}n)$ that is connected by a $p_c$-open path to $\partial B(v,2^{\ell+3})$ (the box of sidelength $2 \cdot 2^{\ell+3}$ centered at $v$) contributes to $Z_n(\ell)$. By the FKG inequality and the RSW theorem, then,
\[
\EE Z_n(\ell) \geq c_{10} f(\ell) \pi(2^{\ell+3}n) \#\{v : v \in \text{Ann}(2^\ell n,2^{\ell+1}n)\}.
\]
Here, $c_{10}$ is a lower bound for the probability of existence of the circuit, $f(\ell) > 0$ is a lower bound (depending only on $\ell$) for the probability of a connection between the two boxes, and $\pi(2^{\ell+3}n)$ is the probability corresponding to the connection between $v$ and $\partial B(v,2^{\ell+3}n)$. By \eqref{bounded D_1}, we obtain
\[
\EE Z_n(\ell) \geq \left[ c_{10} \frac{D_1}{\sqrt{2^{\ell+3}}} 2^{2\ell}\right] n^2 \pi(n).
\]
If we fix $\ell = \ell_0$ so large that this is bigger than $2Cn^2\pi(n)$ for all $n$, we obtain from \eqref{eq: implication} and \eqref{eq: PZ} that
\begin{equation}\label{eq: PZ2}
\mathbb{P}(B_n(C)) \geq \frac{C^2(n^2\pi(n))^2}{\EE Z_n(\ell_0)^2}.
\end{equation}

For the upper bound on $\EE Z_n(\ell_0)^2$, we follow the strategy of Kesten in \cite[p. 388-391]{MR859839}. First note that any $v$ counted in $Z_n(\ell_0)$ must have a $p_c$-open path connecting it to $\partial B(v,2^{\ell_0-1}n)$. Therefore by independence,
\begin{align}
\EE Z_n(\ell_0)^2 &\leq \sum_{v,w \in \text{Ann}(2^{\ell_0}n, 2^{\ell_0+1}n)} \mathbb{P}(v \xleftrightarrow{p_c} \partial B(v,2^{\ell_0-1}n), w \xleftrightarrow{p_c} \partial B(w,2^{\ell_0-1}n)) \nonumber \\
&\leq \sum_{v \in \text{Ann}(2^{\ell_0}n,2^{\ell_0+1}n)} \sum_{k=0}^{2^{\ell_0+2}n} \sum_{w : \|v-w\|_\infty = k} \pi(k/2) \pi(k/2) \pi(2k,2^{\ell_0-1}n). \label{eq: cheese_pizza}
\end{align}
Here, $\pi(2k,2^{\ell_0-1}n)$ is the probability that there is an open path connecting $B(2k)$ to $\partial B(2^{\ell_0-1}n)$.  (If $2k \geq 2^{\ell_0-1}n$, this probability is one.) By quasimultiplicativity \cite[Eq.~(4.17)]{MR2438816} and the RSW theorem, we have 
\[
\pi(k/2) \pi(2k,2^{\ell_0}n) \leq C_{23}\pi(2^{\ell_0}n),
\] 
which is itself bounded by $C_{23}\pi(n)$, so putting this in \eqref{eq: cheese_pizza}, we have an upper bound
\[
\EE Z_n(\ell_0)^2 \leq \left[ C_{23} 2^{2\ell_0} \sum_{k=0}^{2^{\ell_0+2}n} \pi(k)\right] n^2 \pi(n).
\]
By \cite[Eq.~(7)]{MR859839}, we have $\sum_{k=0}^{2^{\ell_0+2}n} \pi(k) \leq C_{24} 2^{2(\ell_0+2)}n^2 \pi(n)$, and so we finish with $\EE Z_n(\ell_0)^2 \leq C_{25}(n^2 \pi(n))^2$, where $C_{25}$ depends only on $\ell_0$. Putting this into \eqref{eq: PZ2} finishes the proof of \eqref{eq: step_4_again}.
\end{proof}

\bigskip
\noindent
{\bf Acknowledgements} The research of M.D. is supported by an NSF CAREER grant.

\bibliographystyle{amsplain}
\bibliography{ref}

\clearpage



\end{document}